\documentclass[a4paper]{amsart}
\usepackage{mathmacro}
\title[Logarithmic Sobolev inequalities via semigroup interpolation]{Sharp quantum logarithmic Sobolev inequalities via semigroup interpolation}
\hypersetup{hidelinks,pdftitle={Sharp quantum logarithmic Sobolev inequalities via semigroup interpolation}}
\author{Gan Yao}
\address[Gan Yao]{Institute for Advanced Study in Mathematics, Harbin Institute of Technology,  Harbin 150001, China.}
\email{gan.yao.math@gmail.com}
\date{\today}

\begin{document}
\begin{abstract}
We develop a semigroup interpolation method for sharp logarithmic Sobolev inequalities using spectral calculus and estimates of $\operatorname{Re}\inner{Lx,x^2}_2$, where $L$ generates an auxiliary symmetric quantum Markov semigroup. We apply the method to free groups and to crossed products of a finite von Neumann algebra $\mathcal N$ by $\mathbb Z$ or $\mathbb Z_n$, for $n=2$ or $n\ge4$, when the action commutes with the semigroup on $\mathcal N$. In particular, the word-length generators on quantum tori and the Weyl word-length generators on $M_n(\mathbb C)$ satisfy logarithmic Sobolev inequalities with optimal constant $2$.
\end{abstract}
\maketitle

\section{Introduction}\label{Sec: Introduction}
Hypercontractivity emerged from harmonic analysis and quantum field theory. Foundational results include Bonami's inequalities for Fourier multipliers on product spaces~\cite{MR283496} and Nelson's work on the classical Ornstein--Uhlenbeck semigroup~\cite{MR343816}. Let $(P_t)_{t\geq0}$ be a Markov semigroup on a probability space $(\Omega,\Sigma,\mu)$. For $1<p<q<\infty$, hypercontractivity means that there exists a time $t_{p,q}$ such that
\[
\norm{P_tf}_q\leq\norm f_p,\qquad t\geq t_{p,q},\quad f\in L_p(\Omega,\mu).
\]
The sharp hypercontractivity problem is to determine the least admissible time,
\[
t_{p,q}^{\mathrm{opt}}=\inf\{t\geq0:\norm{P_t}_{L_p\to L_q}\leq1\}.
\]
For a symmetric Markov semigroup $P_t=e^{-tA}$, Gross' equivalence~\cite{MR420249} identifies the family of bounds
\[
\norm{P_tf}_q\le\norm f_p \quad\text{whenever}\quad  t\ge\frac C4\log\frac{q-1}{p-1},\qquad 1<p\le q<\infty,
\]
with the logarithmic Sobolev inequality (LSI)
\[
\Ent_\mu(f^2)\le C\norm{A^{1/2}f}_{L_2(\mu)}^2,\qquad f\in\Dom(A^{1/2})_+.
\]
This equivalence makes the LSI a natural route to hypercontractivity.

Hypercontractivity has been extensively studied for semigroups on manifolds. Weissler~\cite{MR578933} proved optimal hypercontractivity for the heat and Poisson semigroups on the circle $\mathbb S^1$, while Rothaus~\cite{MR593787} independently treated the heat case. For related results on the sphere $\mathbb S^n$, see~\cite{MR674060,MR706641,MR1164616,MR4278128}.

In discrete settings, the absence of a differential chain rule makes sharp estimates particularly delicate. For the uniform two-point space and its product cubes, the work of Bonami~\cite{MR283496}, Nelson~\cite{MR343816}, Gross~\cite{MR420249}, and Beckner~\cite{MR385456} establishes the optimal time $t_{p,q}^{\mathrm{opt}}=\tfrac12\log\bigl(\tfrac{q-1}{p-1}\bigr)$. For the finite cyclic group $\mathbb Z_n$, the case $n=4$ follows from~\cite{MR730056} and Gross' equivalence~\cite{MR420249}; earlier partial results include~\cite{MR1758914, MR1883499, MR2314078, MR3709719}. More recently, Yao~\cite{yao2025optimalhypercontractivitylogsobolevinequalities} first proved the optimal time for $n=2^k$ and $n=3\cdot 2^k$ with $k\ge1$. The optimal time for $n=3$ and general $p,q$ was determined by Cao, Fan, Han, Qiu, and Wang~\cite{cao2026optimalhypercontractiveconstantsmathbbz3}. Yao~\cite{yao2026completesolutionoptimalhypercontractivity} and, independently, Xie and Zhang~\cite{xie2026sharplogsobolevinequalitiesfinite} established the LSI with constant $2$ for every $n\ge4$. Xie and Zhang combine third-moment estimates with a cubic majorant for entropy introduced by Frank and Ivanisvili~\cite{FRANK2026111159}.

The hypercontractivity problem also arises in noncommutative $L_p$ spaces. Earlier work treats free group von Neumann algebras and free products~\cite{MR3709719,MR3474461,MR3377067}, as well as partial results for finite-dimensional matrix algebras~\cite{MR3390901}. Xie and Zhang~\cite{xiezhang2026hypercontractivityfreegroup} recently established the optimal hypercontractive time for the word-length Poisson semigroup on free groups and supplied further examples through a criterion involving scalar third moments. Zhang~\cite{zhang2026sharphypercontractivityfreeorthogonal} proved sharp hypercontractivity for normalized heat semigroups on free orthogonal quantum groups of Kac type using related scalar third-moment estimates.

Our main results concern crossed products, which include quantum tori and full matrix algebras. Let $(\mathcal N,\tau_{\mathcal N})$ be a finite von Neumann algebra with a faithful normal tracial state. For a positive element $h$ in a finite tracial algebra $(\mathcal M,\tau)$ and a nonnegative self-adjoint generator $T$, write
\[
\Ent_\tau(h)=\tau(h\log h)-\tau(h)\log\tau(h),\qquad \mathcal E_T(x)=\norm{T^{1/2}x}_2^2.
\]
We use the convention $0\log0=0$ and define $\Ent_\tau$ on $L_1(\mathcal M,\tau)_+$ with values in $[0,\infty]$. It is lower semicontinuous for $L_1$ convergence.

Let $B\ge0$ be a symmetric completely Markov generator on $\mathcal N$, and let $C_B<\infty$ be the optimal constant in
\begin{equation}\label{Ineq: Base LSI}
\Ent_{\tau_{\mathcal N}}(a^2)\leq C_B\mathcal E_B(a), \qquad a\in\Dom(B^{1/2})_+.
\end{equation}
Suppose that a normal trace-preserving automorphism $\vartheta$ satisfies
\begin{equation}\label{Eqn: Covariance}
\vartheta\Dom(B)=\Dom(B),\qquad B\vartheta=\vartheta B \quad\hbox{on }\Dom(B).
\end{equation}
Let $G=\mathbb Z$ or $G=\mathbb Z_n$ with $n=2$ or $n\ge4$, and assume $\vartheta^n=\id$ in the finite case. We use the standard word length: $\ell(k)=|k|$ on $\mathbb Z$ and $\ell(k)=\ell_n(k)=\min\{k,n-k\}$ for $0\le k<n$ on $\mathbb Z_n$. On $\mathcal M=\mathcal N\rtimes_\vartheta G$, with implementing unitary $U$ and canonical trace $\tau$, define
\begin{equation}\label{Eqn: Crossed generator}
\widehat B(U^ka)=U^kBa,\qquad \mathsf N(U^ka)=\ell(k)U^ka,\qquad A=\widehat B+\mathsf N.
\end{equation}

Our proof combines semigroup interpolation with spectral calculus and cubic estimates. We use the word-length generator $L=\mathsf N$ for the interpolation. For a bounded positive Fourier polynomial $x=\sum_kU^ka_k$, set
\[
x_t=e^{-t\mathsf N}x=\sum_ke^{-t\ell(k)}U^ka_k.
\]
As $t\to\infty$, $x_t$ converges in operator norm to $a_0\in\mathcal N$, the fixed-point algebra of $e^{-t\mathsf N}$. The diagonal action of $\mathsf N$ on the Fourier coefficients allows us to combine spectral calculus with estimates of $\operatorname{Re}\inner{Lx,x^2}_2$ and obtain
\[
\Ent_\tau(x^2)-\sum_k\Ent_{\tau_{\mathcal N}}(|a_k|^2)\le2\mathcal E_{\mathsf N}(x).
\]
Applying the LSI on $\mathcal N$ to the entropy term $\Ent_{\tau_{\mathcal N}}(|a_k|^2)$ yields the inequality below for $A=\widehat B+\mathsf N$. The passage from positive inputs in \eqref{Ineq: Base LSI} to arbitrary $a\in \Dom(B^{1/2})$ is justified in Section~\ref{Sec: Compatible Markov semigroup}. Restrictions to $\mathcal N$ and the group algebra $\mathcal L(G)$ establish the optimal constant.

\begin{theorem}[Sharp LSI on crossed product]\label{Thm: Sharp crossed product LSI}
Let $(\mathcal N,\tau_{\mathcal N})$ be a finite von Neumann algebra with a faithful normal tracial state, and let $G=\mathbb Z$ or $G=\mathbb Z_n$ with $n=2$ or $n\ge4$. Suppose that $B$ is a symmetric completely Markov generator on $\mathcal N$ with optimal LSI constant $C_B<\infty$, and that the normal trace-preserving automorphism $\vartheta$ satisfies~\eqref{Eqn: Covariance}, with $\vartheta^n=\id$ in the finite case. Let $A$ be the generator on $\mathcal M=\mathcal N\rtimes_\vartheta G$ determined by~\eqref{Eqn: Crossed generator}. Then
\begin{equation}\label{Ineq: Crossed LSI with cyclic group}
\Ent_\tau(x^2)\le C_B\mathcal E_{\widehat B}(x)+2\mathcal E_{\mathsf N}(x),\qquad x\in\Dom(A^{1/2})_+.
\end{equation}
In particular, the optimal LSI constant of $A$ is $\max\{C_B,2\}$.
\end{theorem}

The one-dimensional result can be iterated along the coordinates of $\mathbb Z^d$. Let $d\ge1$, and let $\vartheta_1,\ldots,\vartheta_d$ be pairwise commuting normal trace-preserving automorphisms of $\mathcal N$. Assume, for every $1\le j\le d$, that
\begin{equation}\label{Eqn: Pairwise Covariance}
\vartheta_j\Dom(B)=\Dom(B),\qquad B\vartheta_j=\vartheta_jB\quad\text{on }\Dom(B).
\end{equation}
Let $\vartheta$ be the resulting action of $\mathbb Z^d$, let $\sigma$ be a normalized scalar $2$-cocycle, and equip $\mathcal M=\mathcal N\rtimes_{\vartheta,\sigma}\mathbb Z^d$ with its canonical trace. We use the word length associated with the standard coordinate generators, $|k|_1=\sum_{j=1}^d|k_j|$, and define
\begin{equation}\label{Eqn: Twisted Crossed generator}
\widehat B(u_ka)=u_kBa,\qquad \mathsf N(u_ka)=|k|_1u_ka,\qquad A=\widehat B+\mathsf N.
\end{equation}
Iteration of Theorem~\ref{Thm: Sharp crossed product LSI} preserves the two coefficients separately, even in the presence of a scalar cocycle.

\begin{corollary}\label{Cor: Zd crossed product LSI}
Let $(\mathcal N,\tau_{\mathcal N})$ be a finite von Neumann algebra with a faithful normal tracial state, and let $B$ be a symmetric completely Markov generator with optimal LSI constant $C_B<\infty$. Let $\vartheta$ be an action of $\mathbb Z^d$, $d\ge1$, by normal trace-preserving automorphisms satisfying~\eqref{Eqn: Pairwise Covariance}, and let $\sigma$ be a normalized scalar $2$-cocycle. The generator $A$ on $\mathcal M=\mathcal N\rtimes_{\vartheta,\sigma}\mathbb Z^d$, determined by~\eqref{Eqn: Twisted Crossed generator}, satisfies
\begin{equation}\label{Ineq: Zd crossed LSI}
\Ent_\tau(x^2)\leq C_B\mathcal E_{\widehat B}(x)+2\mathcal E_{\mathsf N}(x),\qquad x\in\Dom(A^{1/2})_+.
\end{equation}
In particular, the optimal LSI constant of $A$ is $\max\{C_B,2\}$.
\end{corollary}

For a real skew-symmetric $d\times d$ matrix $\Theta$, let $\mathbb T_\Theta^d$ denote the associated quantum torus, with canonical Fourier unitaries $W_m$ and trace $\tau_\Theta$. Taking $\mathcal N=\mathbb C$, $B=0$, and $\sigma_\Theta(k,\ell)=e^{\pi i k^{\mathrm T}\Theta\ell}$ in Corollary~\ref{Cor: Zd crossed product LSI} yields the following result.

\begin{corollary}[Sharp LSI on quantum tori]\label{Thm: Sharp LSI on quantum tori}
For every integer $d\geq1$ and every real skew-symmetric $d\times d$ matrix $\Theta$, the word-length generator $AW_m=|m|_1W_m$ on the quantum torus $\mathbb T_\Theta^d$ satisfies
\[
\Ent_{\tau_\Theta}(x^2)\leq2\norm{A^{1/2}x}_2^2, \qquad x\in\Dom(A^{1/2})_+.
\]
The constant $2$ is optimal.
\end{corollary}
Khalkhali and Sadeghi~\cite[Theorem~3.2]{MR3684526} proved the corresponding ordinary LSI for positive smooth elements of a two-dimensional quantum torus whose Fourier support lies on a line. Junge, Palazuelos, Parcet and Perrin~\cite[Section~3.3]{MR3709719} obtained sharp $L_2\to L_q$ estimates for even integers $q$ by comparison with the classical torus. For the same word-length generator, Brannan, Gao and Junge~\cite[Corollary~6.4(iii)]{MR4649063} proved a complete modified LSI, which is distinct from the ordinary LSI considered here.\par

For matrix algebra $M_n(\mathbb{C})$ with $n=2$ or $n\geq 4$, take $\mathcal N=\ell_\infty(\mathbb Z_n)\cong\mathcal D_n$, the diagonal algebra in $M_n(\mathbb C)$, and let $\vartheta$ be the cyclic shift. On the standard basis $(e_j)_{j\in\mathbb Z_n}$, put
\[
Ue_j=e_{j+1},\qquad Ve_j=\omega^je_j,\qquad \omega=e^{2\pi i/n}.
\]
Then $UVU^*=\omega^{-1}V$, $\mathcal D_n=W^*(V)$, and the generator on $\mathcal N$ is $B_n(V^\ell)=\ell_n(\ell)V^\ell$. The realization $\mathcal D_n\rtimes_\vartheta\mathbb Z_n\cong M_n(\mathbb C)$ has canonical trace $\operatorname{tr}_n=n^{-1}\operatorname{Tr}$ and orthonormal Weyl basis $(U^kV^\ell)_{k,\ell\in\mathbb Z_n}$. Applying Theorem~\ref{Thm: Sharp crossed product LSI} with this choice of $\mathcal N$ and $B_n$ yields the sharp inequality for the resulting semigroup.
\begin{corollary}[Sharp LSI on matrix algebra]\label{Thm: Sharp LSI on matrix algebra}
For $n=2$ or $n\ge4$, define the Weyl word-length generator $A_n$ on $M_n(\mathbb C)$ by
\[
A_n(U^kV^\ell)=\bigl(\ell_n(k)+\ell_n(\ell)\bigr)U^kV^\ell.
\]
Then
\[
\Ent_{\operatorname{tr}_n}(x^2)\leq2\norm{A_n^{1/2}x}_2^2, \qquad x\in M_n(\mathbb C)_+,
\]
and the constant $2$ is optimal.
\end{corollary}
The case $n=2$ is already covered by King's theorem for unital qubit semigroups~\cite[Theorem~1]{MR3196986}: the nonzero eigenvalues of $A_2$ in the Pauli basis are $1,1,2$. For this same matrix generator with $n\ge4$, Junge, Palazuelos, Parcet and Perrin~\cite[Theorem~B]{MR3390901} proved sharp $L_2\to L_q$ estimates for even integers $q$ when $n$ is even or $n=5$, and when $n\ge7$ is odd and $q\le n$.

The present semigroup interpolation method grows out of the Lagrange-multiplier analysis for cyclic groups in~\cite[Sections~2--3]{yao2026completesolutionoptimalhypercontractivity} and its extension to free groups in~\cite{yao2026sharplogarithmicsobolevinequalities}. Studying the stationarity equations leads to the pairing $\operatorname{Re}\inner{Lx,Ax-x\log x}_2$, which is implicit in the cyclic-group polynomial reduction and appears explicitly in~\cite[Equation~(1.3)]{yao2026sharplogarithmicsobolevinequalities}. The estimates of this pairing used in the variational arguments admit an interpretation through differentiation along $e^{-tL}$, as explained in Section~\ref{Sec: Semigroup interpolation} and Remark~\ref{Rmk: Choice of interpolation}. This interpretation leads to the semigroup formulation developed here and its extension to crossed products.

For classical semigroup proofs of logarithmic Sobolev inequalities, see~\cite[Proposition 5.5.1]{MR3155209}. Related analyses of entropy and Fisher information along the Ornstein--Uhlenbeck density flow appear in Ledoux, Nourdin and Peccati~\cite{MR3320893}, while Fathi, Indrei and Ledoux~\cite{MR3567822} study the logarithmic Sobolev deficit.

For free groups, $A\lambda_g=|g|\lambda_g$ is the word-length generator on $\mathcal L(\mathbb F_n)$, and we take $L=A$. We rewrite the LSI as the nonnegativity of the logarithmic Sobolev deficit $\mathcal{F}_A(x)$
\[
\mathcal F_A(x)=2\mathcal E_A(x)-\Ent_\tau(x^2)\geq 0.
\]
For uniformly positive Fourier polynomials, we prove
\[
\mathcal Q_{A;A}(x)=\operatorname{Re}\inner{Ax,Ax-x\log x}_2+\mathcal E_A(x)\log\norm x_2\geq0
\]
and hence
\[
\frac{\dd}{\dd t}\mathcal F_A(e^{-tA}x)=-4\mathcal Q_{A;A}(e^{-tA}x)\leq0.
\]
The main estimate concerns $\operatorname{Re}\inner{Ax,x\log x}_2$: its spectral representation and concavity of the logarithm reduce the argument to a bound for
\[
\operatorname{Re}\inner{Au,u^2}_2,\qquad u=x-\tau(x)\mathbf1.
\]
This gives an independent proof of the sharp free-group LSI also established in~\cite{xiezhang2026hypercontractivityfreegroup}. The cubic-majorant method of Frank and Ivanisvili~\cite{FRANK2026111159}, used by Xie and Zhang~\cite{xie2026sharplogsobolevinequalitiesfinite,xiezhang2026hypercontractivityfreegroup} and Zhang~\cite{zhang2026sharphypercontractivityfreeorthogonal}, estimates $\tau(u^3)$ for $\tau(u)=0$. Our estimates retain the factors contributed by the interpolation generator and are proved independently of those third-moment bounds.

For crossed products, we interpolate with $L=\mathsf N$ to obtain the LSI for $A=\widehat B+\mathsf N$. We adapt the operator-valued transference of Chen, Xu and Yin~\cite[Section~2.3]{MR3079331} to represent the weighted sum of entropies $\frac12\sum_k\ell(k)\Ent_{\tau_{\mathcal N}}(|a_k|^2)$ and estimate $\operatorname{Re}\inner{\mathsf Nx,x\log x}_2$ on a common Hilbert space. A resolvent comparison reduces the difference of the resulting logarithmic forms to a scalar cubic estimate on $\widehat G$.

The choice of interpolation is not unique. Remark~\ref{Rmk: Choice of interpolation} discusses the nearest-neighbor generator $L_{\mathrm{nn}}=2I-S_n-S_n^*$ on scalar cyclic groups used in~\cite{yao2026completesolutionoptimalhypercontractivity}, and the averaging generator $L_0=I-E_\tau$, where $E_\tau(x)=\tau(x)\mathbf1$. In both comparisons, the operator $A$ remains the generator of the LSI under consideration, while $L_{\mathrm{nn}}$ or $L_0$ specifies the interpolation. The comparison using $L_0$ explains a consequence of the third-moment estimates of Xie and Zhang and is not used in our main proofs.

The article is organized as follows. Section~\ref{Sec: Crossed preliminaries} fixes the analytic conventions, develops trace differentiation and the spectral representation for derivations, and states the semigroup interpolation scheme. Section~\ref{Sec: Free group case} applies this scheme to prove the sharp LSI on free groups and discusses the choice of interpolation in Remark~\ref{Rmk: Choice of interpolation}. Section~\ref{Sec: Integer crossed products} treats integer and finite-cyclic crossed products in a common framework, then derives the sharp LSI for scalar twisted crossed products by $\mathbb Z^d$, quantum tori, and matrix algebras. Section~\ref{Sec: Cubic pairing estimates} proves the cubic estimates involving the generators used in the preceding arguments.

\section{Preliminaries}\label{Sec: Crossed preliminaries}

\subsection{Tracial von Neumann algebras, \texorpdfstring{$L_p$}{Lp} spaces, and crossed products}

Let $(\mathcal M,\tau)$ be a finite von Neumann algebra with a faithful normal tracial state. For $1\le p<\infty$, define $L_p(\mathcal M,\tau)$ as the completion of $\mathcal M$ with respect to
\[
\norm x_p=\tau(|x|^p)^{1/p},\qquad x\in\mathcal M,
\]
where $|x|=(x^*x)^{1/2}$. Its elements can be identified with possibly unbounded $\tau$-measurable operators affiliated with $\mathcal M$. We put $L_\infty(\mathcal M,\tau)=\mathcal M$ with the operator norm and abbreviate these spaces to $L_p$ when the algebra and trace are clear. We use inner products that are linear in the second variable; in particular,
\[
\inner{x,y}_2=\tau(x^*y),\qquad x,y\in L_2.
\]
We refer to~\cite{MR1999201} for these facts and further background on noncommutative $L_p$ spaces.\par

For a discrete group $G$ with identity $e$, the left regular representation acts on $\ell_2(G)$ by $\lambda_g\delta_h=\delta_{gh}$, where $(\delta_h)_{h\in G}$ is the canonical orthonormal basis. The group von Neumann algebra is
\[
\mathcal L(G)=\overline{\operatorname{span}}^{\,w^*}\{\lambda_g:g\in G\} \subset B(\ell_2(G)).
\]
Its canonical faithful normal tracial state is given by
\[
\tau(x)=\inner{\delta_e,x\delta_e}_{\ell_2(G)},\qquad \tau(\lambda_g)=\begin{cases}1,&g=e,\\0,&g\ne e.\end{cases}
\]
For $f\in\ell_1(G)$, the series $\lambda(f)=\sum_{g\in G}f(g)\lambda_g$ converges in operator norm and satisfies $\tau(\lambda(f))=f(e)$. If $G$ is abelian, the Fourier transform identifies $\mathcal L(G)$ with $L_\infty(\widehat G)$ and its noncommutative $L_p$ spaces with the classical $L_p$ spaces on the compact Pontryagin dual $\widehat G$, equipped with normalized Haar measure.

Let $G=\mathbb Z$ or $G=\mathbb Z_n$, $n\ge2$, and let $\vartheta$ be a normal trace-preserving automorphism of $(\mathcal N,\tau_{\mathcal N})$, with $\vartheta^n=\id$ in the finite case. We recall the regular representation of $\mathcal M=\mathcal N\rtimes_\vartheta G$; see~\cite[Section~X.2]{MR1943006}. For a faithful normal representation $\mathcal N\subset B(H)$, set
\[
(\pi(a)\xi)_j=\vartheta^{-j}(a)\xi_j,\qquad (U\xi)_j=\xi_{j-1},\qquad \xi\in\ell_2(G;H),\quad j\in G.
\]
Then $\mathcal M=W^*(\pi(\mathcal N),U)$. Identifying $\pi(\mathcal N)$ with $\mathcal N$, we have $U^kaU^{-k}=\vartheta^k(a)$. For $G=\mathbb Z_n$, all indices and sums are taken modulo $n$, and $U^n=\mathbf1$. The canonical trace on $\mathcal{M}$ is
\[
\tau(U^ka)=\mathbf1_{\{k=0\}}\tau_{\mathcal N}(a),
\]
and $L_2(\mathcal{M})$ can be decomposed as
\[
L_2(\mathcal M)=\bigoplus_{k\in G}U^kL_2(\mathcal N).
\]
Thus every $x\in L_2(\mathcal M)$ has a unique expansion $x=\sum_{k\in G}U^ka_k$ converging in $L_2$, with coefficients written on the right. The Fourier orthogonality gives
\[
\norm x_2^2=\sum_{k\in G}\norm{a_k}_2^2.
\]

We identify $\widehat G$ with $\mathbb{T}$ for $G=\mathbb Z$, and with $\{2\pi j/n:0\le j<n\}$ for $G=\mathbb Z_n$. Its normalized Haar measure is respectively $\dd m(\theta)=\frac{\dd\theta}{2\pi}$ or $m=\frac{1}{n}\sum_{j=0}^{n-1}\delta_{2\pi j/n}$. With $\chi_k(\theta)=e^{ik\theta}$, the dual action and Fourier projections are
\[
\alpha_\theta(U^ka)=\chi_k(\theta)U^ka,\qquad P_kx=\int_{\widehat G}\overline{\chi_k(\theta)}\alpha_\theta(x)\,\dd m(\theta)=U^ka_k.
\]
Here $P_k:L_2(\mathcal M)\to U^kL_2(\mathcal N)$ is the orthogonal projection onto $U^kL_2(\mathcal N)$. Thus $P_k^*=P_k$, $P_kP_j=0$ for $k\ne j$, and $\sum_kP_k=I$ strongly on $L_2(\mathcal M)$. Averaging the dual action gives the trace-preserving conditional expectation onto $\mathcal N$:
\begin{equation}\label{Eqn: Base expectation}
P_0=\mathbb E_{\mathcal N},\qquad \mathbb E_{\mathcal N}x=\int_{\widehat G}\alpha_\theta(x)\,\dd m(\theta)=a_0, \qquad a_k=\mathbb E_{\mathcal N}(U^{-k}x).
\end{equation}

\subsection{Markov semigroups, quadratic forms, and word-length generators}\label{Sec: Compatible Markov semigroup}

A symmetric completely Markov generator $T$ is a nonnegative self-adjoint operator on $L_2(\mathcal M)$ whose semigroup $P_t=e^{-tT}$ restricts to normal, unital, completely positive and trace-preserving maps on $\mathcal M$. Symmetry means $\inner{P_tu,v}_2=\inner{u,P_tv}_2$, and $P_t$ is strongly continuous on $L_2$. In particular, $T\mathbf1=0$. We use the closed quadratic form
\[
\mathcal E_T(x)=\norm{T^{1/2}x}_2^2,\qquad x\in\Dom(T^{1/2}).
\]
When $T=\delta^*\delta$ for a derivation $\delta$, the semigroup gives the interpolation, while the Leibniz rule gives the spectral representation used below.

The following inequality will be used with Fourier coefficients. Extending $\mathcal E_T$ to $+\infty$ outside its domain, spectral calculus gives
\[
\mathcal E_T(z)=\lim_{t\downarrow0}\frac1t \inner{z,(I-P_t)z}_2,\qquad z\in L_2.
\]
For self-adjoint $z\in\Dom(T^{1/2})$, write $z=z_+-z_-$. Positivity and symmetry of $P_t$ imply
\[
\inner{|z|,P_t|z|}_2-\inner{z,P_tz}_2 =4\inner{z_+,P_tz_-}_2\ge0.
\]
Since $\norm{|z|}_2=\norm z_2$, the limit formula yields
\begin{equation}\label{eq:modulus-form}
|z|\in\Dom(T^{1/2}),\qquad \mathcal E_T(|z|)\le\mathcal E_T(z).
\end{equation}
For arbitrary $z\in\Dom(T^{1/2})$, complete positivity of $P_t$ allows us to apply the same argument to the self-adjoint matrix $\left(\begin{smallmatrix}0&z\\z^*&0\end{smallmatrix}\right)$. Since $\mathcal E_T(z^*)=\mathcal E_T(z)$, we obtain
\[
|z|,|z^*|\in\Dom(T^{1/2}),\qquad \mathcal E_T(|z|)+\mathcal E_T(|z^*|)\le 2\mathcal E_T(z).
\]
For $a\in\Dom(B^{1/2})$, polar decomposition and traciality give $\Ent_{\tau_{\mathcal N}}(|a|^2)=\Ent_{\tau_{\mathcal N}}(|a^*|^2)$. Averaging the LSI on $\mathcal N$ in~\eqref{Ineq: Base LSI} for $|a|$ and $|a^*|$, and applying the preceding estimate with $T=B$, gives
\begin{equation}\label{Ineq: Base LSI for general element}
\Ent_{\tau_{\mathcal N}}(|a|^2)\le C_B\mathcal E_B(a),\qquad a\in\Dom(B^{1/2}).
\end{equation}

For the standard word length
\[
\ell(k)=
\begin{cases}
|k|,&G=\mathbb Z,\\
\min\{k,n-k\},\quad 0\le k<n,&G=\mathbb Z_n,
\end{cases}
\]
the word-length generator and its semigroup on $L_2(\mathcal M)$ are
\begin{equation}\label{Eqn: Poisson semigroup}
\mathsf NP_k=\ell(k)P_k,\qquad x_t=e^{-t\mathsf N}x=\sum_{k\in G}e^{-t\ell(k)}P_kx.
\end{equation}
For $t>0$, define the kernel of $e^{-t\mathsf{N}}$ on $\widehat G$ by
\[
p_t(\theta)=\sum_{k\in G}e^{-t\ell(k)}\overline{\chi_k(\theta)},\qquad \int_{\widehat G}\chi_k(\theta)p_t(\theta)\,dm(\theta)=e^{-t\ell(k)}.
\]
The series is absolutely convergent when $G=\mathbb Z$ and finite when $G=\mathbb Z_n$, and $m$ is normalized Haar measure on $\widehat G$. Since $\alpha_\theta(U^ka)=\chi_k(\theta)U^ka$, averaging each Fourier component gives
\begin{equation}\label{Eqn: Crossed word length jumps}
e^{-t\mathsf N}x=\int_{\widehat G}\alpha_\theta(x)p_t(\theta)\,\dd m(\theta),\qquad x\in L_2(\mathcal M),
\end{equation}
where the integral is understood in $L_2$. Let $\nu$ be the positive symmetric measure on $\widehat G\setminus\{0\}$ characterized by
\begin{equation}\label{Eqn: Crossed jump measure}
\int_{\widehat G}f(\theta)\,\dd\nu(\theta)=\lim_{t\downarrow0}\frac1{2t}\int_{\widehat G}f(\theta)p_t(\theta)\,\dd m(\theta)
\end{equation}
for every nonnegative trigonometric polynomial $f$ with $f(0)=0$, and set $\nu(\{0\})=0$. Existence follows by summing the geometric series in the integer case and by a finite Fourier calculation in the cyclic case. For a finite Fourier sum $x$, the generator $\mathsf{N}$ has the representation
\[
\mathsf Nx=\lim_{t\downarrow0}\frac{x-e^{-t\mathsf N}x}{t}=\int_{\widehat G}(2x-\alpha_\theta(x)-\alpha_{-\theta}(x))\,\dd\nu(\theta).
\]
The limit is in $L_2$, and the integral is justified coefficientwise by the identities
\begin{equation}\label{Eqn: Crossed jump identities}
\begin{aligned}
\int_{\widehat G}|1-\chi_k(\theta)|^2\,d\nu(\theta)&=\ell(k),\\
\int_{\widehat G}(1-\chi_r(\theta))(1-\chi_s(\theta))\,d\nu(\theta) &=\frac{\ell(r)+\ell(s)-\ell(r+s)}2\ge0.
\end{aligned}
\end{equation}
By Fourier orthogonality, for a finite Fourier sum $x$, we have
\begin{equation}\label{Eqn: Crossed jump form}
\mathcal E_{\mathsf N}(x) =\int_{\widehat G}\norm{x-\alpha_\theta(x)}_2^2\,d\nu(\theta).
\end{equation}
These identities supply the derivation used in Section~\ref{Sec: Main proof}. \par 

The following lemma shows that $A=\widehat B+\mathsf N$ is a symmetric completely Markov generator under the covariance condition~\eqref{Eqn: Covariance}.
\begin{lemma}\label{Lem: Compatible generator}
Let $B$ be a symmetric completely Markov generator on $L_2(\mathcal N)$ satisfying~\eqref{Eqn: Covariance}. For $G=\mathbb Z$ or $G=\mathbb Z_n$, $n\ge2$, the direct-sum operator $A(U^ka)=U^k(B+\ell(k))a$ is symmetric completely Markov, with
\begin{equation}\label{Eqn: Compatible Markov semigroup}
e^{-tA}(U^ka)=e^{-t\ell(k)}U^ke^{-tB}(a),\qquad t\ge0.
\end{equation}
Its operator domain is
\[
\Dom(A)=\left\{x=\sum_{k\in G}U^ka_k\in L_2(\mathcal M): a_k\in\Dom(B),\ \sum_{k\in G}\norm{(B+\ell(k))a_k}_2^2<\infty\right\}.
\]
For every $x=\sum_{k\in G}U^ka_k\in\Dom(A^{1/2})$, we have
\begin{equation}\label{Eqn: Crossed form}
\begin{aligned}
\norm x_2^2&=\sum_{k\in G}\norm{a_k}_2^2,\\
\mathcal E_A(x)&=\sum_{k\in G}\bigl(\mathcal E_B(a_k)+\ell(k)\norm{a_k}_2^2\bigr) =\mathcal E_{\widehat B}(x)+\mathcal E_{\mathsf N}(x).
\end{aligned}
\end{equation}
\end{lemma}
\begin{proof}
The semigroup $e^{-t\mathsf N}$ is symmetric completely Markov by averaging against the Poisson kernel for $G=\mathbb Z$, and by exponentiating the bounded symmetric jump generator for $G=\mathbb Z_n$. Commutation with $\vartheta$ allows $e^{-tB}$ to extend coefficientwise to a symmetric completely Markov semigroup on $\mathcal M$; see~\cite[Section~2.1]{MR3283931}. In our right-coefficient convention, its action is
\[
e^{-t\widehat B}(U^ka)=e^{-tB}(\vartheta^k(a))U^k=U^ke^{-tB}(a).
\]
This extension commutes with $e^{-t\mathsf N}$. Their composition has the Fourier action~\eqref{Eqn: Compatible Markov semigroup} and generator $A=\widehat B+\mathsf N$, with the stated direct-sum domain. Parseval's identity gives the norm formula in~\eqref{Eqn: Crossed form}, and the direct-sum functional calculus gives the quadratic-form formula on $\Dom(A^{1/2})$.
\end{proof}

\subsection{Spectral calculus and spectral representation for derivations}
Let $S,T$ be self-adjoint operators on a Hilbert space $\mathcal K$, with spectral resolutions $E_S,E_T$. They are said to \emph{strongly commute} if every spectral projection of $S$ commutes with every spectral projection of $T$. In this case, there is a unique joint projection-valued measure $P$ on $\mathbb R^2$, characterized by $P(B\times C)=E_S(B)E_T(C)$ for Borel sets $B,C\subset\mathbb R$. The joint functional calculus is defined, for a Borel function $F:\mathbb R^2\to\mathbb C$, by
\[
F(S,T)=\int_{\mathbb R^2}F(s,t)\,\mathrm dP(s,t).
\]
For $\xi\in\mathcal K$, the scalar measure $\mu_\xi(\Omega)=\inner{\xi,P(\Omega)\xi}_{\mathcal K}$ is positive and has mass $\norm\xi_{\mathcal K}^2$. For bounded Borel $F$,
\begin{equation}\label{eqn: joint spectral probability measure}
\inner{\xi,F(S,T)\xi}_{\mathcal K}=\int F\,\mathrm d\mu_\xi,\qquad \norm{F(S,T)\xi}_{\mathcal K}^2=\int|F|^2\,\mathrm d\mu_\xi.
\end{equation}
For unbounded Borel $F$, finiteness of the second integral is equivalent to $\xi\in\Dom(F(S,T))$, and both identities remain valid on this domain. If $F$ is real and $\int|F|\,\mathrm d\mu_\xi<\infty$, the first identity also holds in the quadratic-form sense: its left-hand side is interpreted as
\[
\lim_{m\to\infty}\inner{\xi,F_m(S,T)\xi}_{\mathcal K},\qquad F_m=\max\{-m,\min\{F,m\}\}.
\]
Dominated convergence identifies this limit with $\int F\,\mathrm d\mu_\xi$; square integrability is needed only to interpret $F(S,T)\xi$ as a vector. We refer to \cite{MR2953553} for further background on the spectral theory of commuting normal operators.

For a real $C^1$ function $f$, write
\[
f^{[1]}(s,t)=
\begin{cases}
\dfrac{f(s)-f(t)}{s-t},&s\ne t,\\[4pt]
f'(s),&s=t.
\end{cases}
\]
We use derivations with values in Hilbert bimodules, as in~\cite[Section~2]{MR2470111}: the left and right actions commute, and the derivation satisfies the Leibniz rule with respect to these actions. We need the following spectral representation only for bounded elements in a unital $*$-algebra contained in the operator domain.

\begin{lemma}\label{Lem: Endpoint representation}
Let $\mathcal H$ be a Hilbert $\mathcal M$-bimodule with commuting normal unital $*$-representations $\pi_L$ of $\mathcal M$ and $\pi_R$ of $\mathcal M^{\mathrm{op}}$. Let $\delta$ be a closed densely defined operator from $L_2(\mathcal M)$ to $\mathcal H$, and put $L=\delta^*\delta$. Suppose that a unital $*$-subalgebra $\mathcal A_0\subset\mathcal M\cap\Dom(L)$ satisfies
\[
\delta(ab)=\pi_L(a)\delta(b)+\pi_R(b)\delta(a), \qquad a,b\in\mathcal A_0.
\]
For every $x=x^*\in\mathcal A_0$, there is a finite positive measure $\mu_x$ on $\sigma(x)^2$ such that
\begin{equation}\label{Eqn: Endpoint representation}
\operatorname{Re}\inner{Lx,f(x)}_2 =\int_{\sigma(x)^2}f^{[1]}(s,t)\,\mathrm d\mu_x(s,t)
\end{equation}
for every real $C^1$ function $f$ on an interval containing $\sigma(x)$. In particular,
\begin{equation}\label{Eqn: Endpoint mass and cubic moment}
\mu_x(\sigma(x)^2)=\inner{Lx,x}_2,\qquad \int_{\sigma(x)^2}(s+t)\,\mathrm d\mu_x(s,t) =\operatorname{Re}\inner{Lx,x^2}_2.
\end{equation}
\end{lemma}
\begin{proof}
Let $P_x$ be the joint spectral measure of the commuting bounded self-adjoint operators $\pi_L(x),\pi_R(x)$, and set
\[
\mu_x(\Omega)=\inner{\delta x,P_x(\Omega)\delta x}_{\mathcal H}.
\]
The Leibniz rule gives $\delta(x^m)=\sum_{k=0}^{m-1}\pi_L(x)^k\pi_R(x)^{m-1-k}\delta x$. Thus, for every polynomial $p$,
\begin{equation}\label{Eqn: Derivation polynomial chain rule}
\delta p(x)=p^{[1]}(\pi_L(x),\pi_R(x))\delta x.
\end{equation}
On a compact interval containing $\sigma(x)$, approximate $f$ and $f'$ uniformly by polynomials and their derivatives. The corresponding divided differences converge uniformly as well, so closedness of $\delta$ gives $f(x)\in\Dom(\delta)$ and
\begin{equation}\label{Eqn: Endpoint chain rule}
\delta f(x)=f^{[1]}(\pi_L(x),\pi_R(x))\delta x.
\end{equation}
Pairing with $\delta x$ and using $L=\delta^*\delta$ proves \eqref{Eqn: Endpoint representation}. Taking $f(s)=s$ and $f(s)=s^2$ gives~\eqref{Eqn: Endpoint mass and cubic moment}.
\end{proof}

\subsection{Semigroup interpolation and spectral calculus framework for logarithmic Sobolev inequalities}\label{Sec: Semigroup interpolation}
We relate the pairing $\operatorname{Re}\inner{Lx,Ax-x\log x}_2$ from the Lagrange-multiplier arguments in~\cite{yao2026completesolutionoptimalhypercontractivity,yao2026sharplogarithmicsobolevinequalities} to the derivative of the logarithmic Sobolev deficit along $e^{-tL}x$. Spectral calculus then reduces the required logarithmic estimate to a bound for $\operatorname{Re}\inner{Lx,x^2}_2$. In the crossed-product setting, we modify the logarithmic Sobolev deficit by adding the sum of entropies $\sum_k\Ent_{\tau_{\mathcal N}}(|a_k|^2)$.

Let $A$ and $L$ be strongly commuting symmetric completely Markov generators on $L_2(\mathcal M,\tau)$. The generator $A$ determines the desired LSI and the generator $L$ determines the interpolation semigroup:
\[
\Ent_\tau(x^2)\le2\mathcal E_A(x),\qquad x_t=e^{-tL}x.
\]
Define the logarithmic Sobolev deficit by
\begin{equation}\label{Eqn: General deficit}
\mathcal F_A(x)=2\mathcal E_A(x)-\Ent_\tau(x^2),\qquad x\in\mathcal M_+\cap\Dom(A^{1/2}),
\end{equation}
and, for nonzero bounded positive $x\in\Dom(A)\cap\Dom(L)$, set
\begin{equation}\label{Eqn: General comparison}
\mathcal Q_{A;L}(x)=\operatorname{Re}\inner{Lx,Ax-x\log x}_2+\mathcal E_L(x)\log\norm x_2.
\end{equation}
Both functionals are homogeneous of degree two and related by the derivative of $\mathcal F_A(x)$ along $x_t=e^{-tL}x$.

\begin{proposition}\label{Prop: General interpolation}
Let $A,L$ be as above. Suppose that $\mathcal A_0\subset\mathcal M\cap\Dom(A)\cap\Dom(L)$ is an $L_2$-dense unital $*$-subalgebra. For $x\in\mathcal A_0$ with $x\ge\varepsilon\mathbf1>0$, write $x_t=e^{-tL}x$, and let $\mathbb{E}$ be the trace-preserving conditional expectation onto the fixed-point algebra of $e^{-tL}$. Then
\begin{equation}\label{Eqn: Deficit derivative comparison}
\frac{\dd}{\dd t}\mathcal F_A(x_t)=-4\mathcal Q_{A;L}(x_t),
\end{equation}
and $\mathcal F_A(x_t)\to\mathcal F_A(\mathbb{E}x)$ as $t\to\infty$. Consequently,
\begin{equation}\label{Eqn: General integrated deficit}
\mathcal F_A(x)=\mathcal F_A(\mathbb{E}x)+4\int_0^\infty\mathcal Q_{A;L}(x_t)\,\dd t,
\end{equation}
where the integral is understood as an improper integral. In particular, if $\mathcal Q_{A;L}(x_t)\ge0$ for all $t\ge0$ and $\mathcal F_A(\mathbb{E}x)\ge0$, then $\Ent_\tau(x^2)\le2\mathcal E_A(x)$.
\end{proposition}
\begin{proof}
Strong commutation gives $x_t\in\Dom(A)\cap\Dom(L)$, and $x_t'=-Lx_t$ in $L_2$. Joint spectral calculus and the trace chain rule, justified by the Markov bounds $\varepsilon\mathbf1\le x_t\le\norm{x}_\infty\mathbf1$, give
\[
\frac{\dd}{\dd t}\mathcal F_A(x_t)=-4\operatorname{Re}\inner{Lx_t,Ax_t-x_t\log x_t}_2-4\mathcal E_L(x_t)\log\norm{x_t}_2=-4\mathcal Q_{A;L}(x_t),
\]
with the right derivative at $t=0$.

Joint spectral calculus also gives $x_t\to\mathbb{E}x$ in the form norm of $A$. The common spectral bounds imply $\Ent_\tau(x_t^2)\to\Ent_\tau((\mathbb{E}x)^2)$, so $\mathcal F_A(x_t)\to\mathcal F_A(\mathbb{E}x)$. Integrating~\eqref{Eqn: Deficit derivative comparison} on $[0,T]$ and letting $T\to\infty$ proves~\eqref{Eqn: General integrated deficit}.
\end{proof}

To estimate the logarithmic pairing $\operatorname{Re}\inner{Lx,x\log x}_2$ in $\mathcal Q_{A;L}(x)$, suppose that $L=\delta^*\delta$ satisfies the hypotheses of Lemma~\ref{Lem: Endpoint representation}. Take a uniformly positive $x\in\mathcal A_0\cap\Dom(A)$ with $\mathcal E_L(x)>0$ and apply Lemma~\ref{Lem: Endpoint representation} to $f(s)=s\log s$. Since
\[
f^{[1]}(s,t)=1+\int_0^1\log((1-v)s+vt)\,\dd v,
\]
concavity of the logarithm, integrated against the probability measure $\mathcal E_L(x)^{-1}\mu_x\otimes \dd v$, gives
\begin{equation}\label{Ineq: General spectral logarithmic bound}
\operatorname{Re}\inner{Lx,x\log x}_2\le\mathcal E_L(x)\left(1+\log\frac{\operatorname{Re}\inner{Lx,x^2}_2}{2\mathcal E_L(x)}\right).
\end{equation}
The ratio in the logarithm is positive because $x\ge\varepsilon\mathbf1$. It is therefore enough to prove the cubic estimate
\begin{equation}\label{Ineq: General sufficient cubic bound}
\frac{\operatorname{Re}\inner{Lx,x^2}_2}{2\mathcal E_L(x)}\le\norm x_2\exp\left(\frac{\operatorname{Re}\inner{Lx,Ax}_2}{\mathcal E_L(x)}-1\right),
\end{equation}
which implies $\mathcal Q_{A;L}(x)\ge0$. If $\mathcal E_L(x)=0$, then $Lx=0$ and $\mathcal Q_{A;L}(x)=0$. When the fixed-point algebra of $e^{-tL}$ consists of scalars, as in the free-group case with $L=A$, we have $\mathcal F_A(\mathbb{E}x)=0$. Section~\ref{Sec: Free group case} proves the required cubic estimate for the free-group word-length generator with $L=A$ and hence proves LSI for free groups.

For crossed products by $G=\mathbb Z$ or $G=\mathbb Z_n$, with $n=2$ or $n\ge4$, we take $L=\mathsf N$ and, for a bounded positive Fourier polynomial $x=\sum_kU^ka_k$, consider the deficit
\[
\mathcal F_G(x)=2\mathcal E_{\mathsf N}(x)+\sum_k\Ent_{\tau_{\mathcal N}}(|a_k|^2)-\Ent_\tau(x^2).
\]
For bounded Fourier polynomials with $a_k\in\Dom(B^{1/2})$, nonnegativity of $\mathcal F_G$ and \eqref{Ineq: Base LSI for general element} give
\[
\Ent_\tau(x^2)\le\sum_k\Ent_{\tau_{\mathcal N}}(|a_k|^2)+2\mathcal E_{\mathsf N}(x)\le C_B\mathcal E_{\widehat B}(x)+2\mathcal E_{\mathsf N}(x),
\]
which is the desired LSI for these crossed-product polynomials.

For a uniformly positive Fourier polynomial, write $x_t=e^{-t\mathsf N}x=\sum_ke^{-t\ell(k)}U^ka_k$. Homogeneity of entropy gives
\[
\Ent_{\tau_{\mathcal N}}(|e^{-t\ell(k)}a_k|^2)=e^{-2t\ell(k)}\Ent_{\tau_{\mathcal N}}(|a_k|^2).
\]
Thus Proposition~\ref{Prop: General interpolation}, applied with $A=L=\mathsf N$, yields
\[
\mathcal Q_G(x)=\mathcal Q_{\mathsf N;\mathsf N}(x)+\frac12\sum_k\ell(k)\Ent_{\tau_{\mathcal N}}(|a_k|^2),\qquad \frac{\dd}{\dd t}\mathcal F_G(x_t)=-4\mathcal Q_G(x_t).
\]
Since $x_t\to a_0$ in operator norm, the surviving coefficient entropy cancels the entropy of $a_0^2$, and $\mathcal F_G(x_t)\to0$. It therefore remains to prove $\mathcal Q_G(x)\ge0$ for uniformly positive Fourier polynomials.

To compare the weighted sum of entropies $\frac12\sum_k\ell(k)\Ent_{\tau_{\mathcal N}}(|a_k|^2)$ with $\operatorname{Re}\inner{\mathsf Nx,x\log x}_2$, we work directly on the weighted Hilbert bimodule $L_2(\widehat G,\nu;L_2(\mathcal M))$, using a representation motivated by the operator-valued transference of Chen, Xu and Yin~\cite[Section~2.3]{MR3079331}. For the measure $\nu$ in~\eqref{Eqn: Crossed jump measure}, define the closed operator
\[
\delta:\Dom(\mathsf N^{1/2})\subset L_2(\mathcal M)\longrightarrow L_2(\widehat G,\nu;L_2(\mathcal M)),\qquad (\delta x)(\theta)=x-\alpha_\theta(x).
\]
Equip the target space with the bimodule actions
\[
(a\cdot v\cdot b)(\theta)=a\,v(\theta)\alpha_\theta(b),\qquad a,b\in\mathcal M.
\]
Then $\delta$ satisfies the Leibniz rule on $\mathcal M\cap\Dom(\mathsf N^{1/2})$, and~\eqref{Eqn: Crossed jump form} gives $\delta^*\delta=\mathsf N$. This gives a concrete Hilbert bimodule realization of the Dirichlet form $\mathcal E_{\mathsf N}$, as in~\cite[Section~4.3]{MR2463708}.

For $x\ne P_0x$, we have $\mathcal E_{\mathsf N}(x)>0$. Put $u=\sum_{k\ne0}\norm{a_k}_2\chi_k$ on $\widehat G$, and let $\mathsf N_G\chi_k=\ell(k)\chi_k$. The symmetry $\norm{a_{-k}}_2=\norm{a_k}_2$ makes $u$ real, with zero mean. In Section~\ref{Sec: Main proof}, we represent the weighted sum of entropies and estimate the logarithmic pairing $\operatorname{Re}\inner{\mathsf Nx,x\log x}_2$ by quadratic forms evaluated at $\delta x$ on this bimodule. Comparing these forms by resolvents gives the analogue of~\eqref{Ineq: General spectral logarithmic bound}:
\[
\operatorname{Re}\inner{\mathsf Nx,x\log x}_2-\frac12\sum_k\ell(k)\Ent_{\tau_{\mathcal N}}(|a_k|^2)\le\mathcal E_{\mathsf N}(x)\left[1+\log\left(\norm{a_0}_2+\frac{\inner{\mathsf N_Gu,u^2}_2}{2\mathcal E_{\mathsf N}(x)}\right)\right].
\]
The scalar estimate
\[
\bigl|\inner{\mathsf N_Gu,u^2}_2\bigr|^2\le4\norm u_2^2\left(\norm{\mathsf N_Gu}_2^4-\inner{u,\mathsf N_Gu}_2^2\right)
\]
gives the analogue of~\eqref{Ineq: General sufficient cubic bound},
\[
\norm{a_0}_2+\frac{\inner{\mathsf N_Gu,u^2}_2}{2\mathcal E_{\mathsf N}(x)}\le\norm x_2\exp\left(\frac{\norm{\mathsf Nx}_2^2}{\mathcal E_{\mathsf N}(x)}-1\right),
\]
which implies $\mathcal Q_G(x)\ge0$. If $x=P_0x$, then $\mathcal Q_G(x)=0$ directly. Integrating its nonnegativity gives $\mathcal F_G(x)\ge0$, and the LSI on $\mathcal N$ yields the asserted LSI for $A=\widehat B+\mathsf N$.

\section{Free group case}\label{Sec: Free group case}
Let $\mathbb F_n$ be the free group on $n$ generators, where $n\in\mathbb N=\{1,2,\ldots\}$, and put $\mathcal M=\mathcal L(\mathbb F_n)$, equipped with its canonical tracial state $\tau$. For the reduced word length $|g|$, set
\[
A\lambda_g=|g|\lambda_g,\qquad P_t\lambda_g=e^{-t|g|}\lambda_g.
\]
These multipliers form a symmetric, unital, trace-preserving completely Markov semigroup; see~\cite[Lemmas~1.1--1.2]{MR520930}. We take $L=A$ as the interpolation generator and prove the sharp LSI below by spectral calculus and a direct estimate of $\operatorname{Re}\inner{Ax,x^2}_2$. Xie and Zhang~\cite[Theorem~1.2]{xiezhang2026hypercontractivityfreegroup} established the same LSI using scalar third-moment estimates. Our proof uses the length-weighted pairing and is independent of those moment bounds.

\begin{theorem}\label{Thm: Free group LSI}
For every $x\in\mathbb C[\mathbb F_n]_+$, we have
\begin{equation}\label{Ineq: LSI for free group}
\Ent_\tau(x^2)\le 2\mathcal E_A(x).
\end{equation}
The constant $2$ is optimal.
\end{theorem}
We apply Section~\ref{Sec: Semigroup interpolation} with $L=A$, where $A\lambda_g=|g|\lambda_g$. Thus
\[
\begin{aligned}
\mathcal F_A(x)&=2\mathcal E_A(x)-\Ent_\tau(x^2),\\
\mathcal Q_{A;A}(x)&=\operatorname{Re}\inner{Ax,Ax-x\log x}_2+\mathcal E_A(x)\log\norm x_2.
\end{aligned}
\]
Proposition~\ref{Prop: General interpolation} gives $\frac{\dd}{\dd t}\mathcal F_A(e^{-tA}x)=-4\mathcal Q_{A;A}(e^{-tA}x)$. The fixed-point algebra of $e^{-tA}$ consists of scalars, so the limit of the deficit is zero. Theorem~\ref{Thm: Free group LSI} is proved by showing $\mathcal{Q}_{A;A}(x)\geq 0$.

\subsection{The word-length derivation and its spectral representation}
We use Haagerup's tree construction in \cite[Proof of Lemma~1.2]{MR520930} to construct the derivation. Orient the edges of the Cayley tree of $\mathbb{F}_n$ from $h$ to $hs_j$, where $s_1,\ldots,s_n$ are the generators. Let $\mathsf E_{\mathbb R}$ have orthonormal basis $\varepsilon_{h,j}$ indexed by these edges; reversing an edge changes its sign. Left translation defines $\pi(g)\varepsilon_{h,j}=\varepsilon_{gh,j}$. If $c(g)$ is the signed sum of the edges on the path from $e$ to $g$, cancellation of backtracking edges gives
\begin{equation}\label{Eqn: Free word cocycle}
c(gh)=c(g)+\pi(g)c(h),\qquad \norm{c(g)}^2=|g|.
\end{equation}
Following \cite[Example~2.2, pp.~421--422]{MR2470111}, let $\mathsf E_{\mathbb C}$ be the complexification of $\mathsf E_{\mathbb R}$ and equip $\mathcal H=\mathsf E_{\mathbb C}\otimes L_2(\mathcal M,\tau)$ with the actions
\begin{equation}\label{Eqn: Free bimodule actions}
\lambda_g\cdot(\xi\otimes\lambda_h)=\pi(g)\xi\otimes\lambda_{gh},\qquad (\xi\otimes\lambda_h)\cdot\lambda_k=\xi\otimes\lambda_{hk}.
\end{equation}
These actions commute. The right action is amplified right multiplication, while the unitary $W(\xi\otimes\lambda_h)=\pi(h^{-1})\xi\otimes\lambda_h$ conjugates the left action to amplified left multiplication. Thus both actions extend normally to $\mathcal M$, making $\mathcal H$ a Hilbert $\mathcal M$-bimodule.

Define $\delta$ on $\mathbb C[\mathbb F_n]$ by extending $\delta(\lambda_g)=c(g)\otimes\lambda_g$ linearly. The cocycle identity gives
\[
\delta(vw)=v\cdot\delta w+\delta v\cdot w, \qquad v,w\in\mathbb C[\mathbb F_n],
\]
and
\begin{equation}\label{Eqn: Free derivation form}
\inner{\delta v,\delta w}_{\mathcal H} =\sum_g|g|\,\overline{\widehat v(g)}\widehat w(g) =\inner{Av,w}_2.
\end{equation}
Orthogonality of the Fourier basis shows that $\delta$ is closable. Its closure, denoted by the same symbol, satisfies $\Dom(\delta)=\Dom(A^{1/2})$, $\delta^*\delta=A$, and $\ker\delta=\mathbb C\mathbf1$. In particular, $\mathcal E_A(x)=\norm{\delta x}_{\mathcal H}^2$, in agreement with the derivation framework of~\cite{MR1986156}. Lemma~\ref{Lem: Endpoint representation} applies with $\mathcal A_0=\mathbb C[\mathbb F_n]$ and $L=A$.

For $x\in\mathbb C[\mathbb F_n]$ with $x\ge\varepsilon\mathbf1>0$, apply the lemma to $f(s)=s$, $f(s)=s^2$, and $f(s)=s\log s$ on the compact spectral interval of $x$. This gives
\begin{equation}\label{Eqn: Free spectral comparison}
\begin{split}
\inner{Ax,x}_2&=\int_{\sigma(x)^2}\mathrm d\mu_x(s,t),\qquad \operatorname{Re}\inner{Ax,x^2}_2=\int_{\sigma(x)^2}(s+t)\,\mathrm d\mu_x(s,t),\\
\operatorname{Re}\inner{Ax,x\log x}_2 &=\int_{\sigma(x)^2}\int_0^1[1+\log((1-v)s+vt)]\,\mathrm dv\,\mathrm d\mu_x(s,t).
\end{split}
\end{equation}
The last identity follows by integrating the derivative of $s\log s$ along the segment from $s$ to $t$. We next bound the logarithmic pairing in terms of $\operatorname{Re}\inner{Ax,x^2}_2$.

\subsection{Reduction to a polynomial inequality}\label{Sec: Free polynomial reduction}

We estimate $\mathcal Q_{A;A}$ from~\eqref{Eqn: General comparison} on uniformly positive Fourier polynomials. All pairings in this subsection are in $L_2(\mathcal L(\mathbb F_n))$, and all functions of $x$ are defined by bounded spectral calculus.

\begin{lemma}\label{Lem: Free endpoint transform}
For nonconstant $x\in\mathbb C[\mathbb F_n]$ with $x\ge\varepsilon\mathbf1>0$, put
\[
z=\frac{\operatorname{Re}\inner{Ax,x^2}_2}{2\inner{Ax,x}_2}.
\]
Then $z>0$ and
\begin{equation}\label{Ineq: Free endpoint gap}
\operatorname{Re}\inner{Ax,x^2-2x\log x}_2 \ge2\inner{Ax,x}_2(z-1-\log z).
\end{equation}
Equivalently,
\begin{equation}\label{Eqn: Free endpoint bound}
\operatorname{Re}\inner{Ax,x\log x}_2 \le\inner{Ax,x}_2(1+\log z).
\end{equation}
Consequently,
\begin{equation}\label{Ineq: Free transformed comparison}
\mathcal Q_{A;A}(x) \ge\inner{Ax,x}_2\left( \frac{\norm{Ax}_2^2}{\inner{Ax,x}_2}-1-\log z+\log\norm x_2 \right).
\end{equation}
\end{lemma}

\begin{proof}
Since $x$ is nonconstant, $\inner{Ax,x}_2>0$. The measure $\mu_x$ is supported in $[\varepsilon,\norm x_\infty]^2$, and therefore
\begin{equation}\label{Eqn: Free cubic positivity}
\operatorname{Re}\inner{Ax,x^2}_2 =\int_{\sigma(x)^2}(s+t)\,\mathrm d\mu_x(s,t) \ge2\varepsilon\inner{Ax,x}_2>0.
\end{equation}
In particular, $z\ge\varepsilon>0$. Put $\varphi(s)=s-1-\log s$ and $f(s)=s^2-2s\log s$. Since $f'=2\varphi$, Lemma~\ref{Lem: Endpoint representation} gives
\[
\operatorname{Re}\inner{Ax,f(x)}_2 =2\int_{\sigma(x)^2}\int_0^1\varphi((1-v)s+vt)\,\mathrm dv\,\mathrm d\mu_x(s,t).
\]
The measure $\mu_x\otimes\mathrm dv/\inner{Ax,x}_2$ is a probability measure. Jensen's inequality for the convex function $\varphi$ gives
\[
\frac{\operatorname{Re}\inner{Ax,f(x)}_2}{2\inner{Ax,x}_2}\ge\varphi\left( \int_{\sigma(x)^2}\int_0^1((1-v)s+vt)\frac{\mathrm dv\,\mathrm d\mu_x(s,t)}{\inner{Ax,x}_2}\right)=\varphi\left(\frac{1}{2\inner{Ax,x}_2} \int_{\sigma(x)^2}(s+t)\,\mathrm d\mu_x(s,t)\right) =\varphi(z).
\]
The last equality follows from~\eqref{Eqn: Free spectral comparison}. This proves~\eqref{Ineq: Free endpoint gap}. Using $\operatorname{Re}\inner{Ax,x^2}_2=2z\inner{Ax,x}_2$ gives \eqref{Eqn: Free endpoint bound}, and substitution into $\mathcal Q_{A;A}(x)$ yields~\eqref{Ineq: Free transformed comparison}.
\end{proof}

Write $S_k=\{g\in\mathbb F_n:|g|=k\}$ and let $\Pi_kv=\sum_{g\in S_k}\widehat v(g)\lambda_g$ be the projection onto the Fourier terms of length $k$. For self-adjoint $x\in\mathbb C[\mathbb F_n]$, write
\begin{equation}\label{Eqn: Free centering}
u=x-\tau(x)\mathbf1,\qquad u_k=\Pi_ku,\qquad b_k=\norm{u_k}_2.
\end{equation}
Since $\tau(u)=0$, the spectral gap of $A$ gives
\begin{equation}\label{Eqn: Free higher energy}
\inner{Au,(A-I)u}_2=\norm{Au}_2^2-\inner{Au,u}_2=\sum_{k\ge1}k(k-1)b_k^2\ge0.
\end{equation}
Since $A\mathbf1=0$ and $\tau(Au)=0$, expanding $x^2=\tau(x)^2\mathbf1+2\tau(x)u+u^2$ gives
\begin{equation}\label{Eqn: Free centered cubic pairing}
\operatorname{Re}\inner{Ax,x^2}_2=2\tau(x)\inner{Au,u}_2+\operatorname{Re}\inner{Au,u^2}_2.
\end{equation}
Moreover, $\inner{Au,u}_2\ge\norm u_2^2>0$ when $x$ is nonconstant.

\begin{lemma}\label{Lem: Free cubic bound}
For every self-adjoint $u\in\mathbb C[\mathbb F_n]$ with $\tau(u)=0$,
\begin{equation}\label{Ineq: Free cubic bound}
\operatorname{Re}\inner{Au,u^2}_2\le2\norm u_2\sqrt{2\inner{Au,u}_2\inner{Au,(A-I)u}_2}+\inner{Au,(A-I)u}_2 \sqrt{\inner{Au,u}_2-\norm u_2^2}.
\end{equation}
\end{lemma}
The proof is given in Section~\ref{sec: weighted cubic estimation for free group}. We first use the estimate to complete the scalar comparison.

\begin{lemma}\label{Lem: Free scalar comparison}
Let $x\in\mathbb C[\mathbb F_n]$ be self-adjoint and nonconstant, and put $u=x-\tau(x)\mathbf1$. Then
\begin{equation}\label{Ineq: Free strict scalar comparison}
z=\frac{\operatorname{Re}\inner{Ax,x^2}_2}{2\inner{Ax,x}_2} <\norm x_2\exp\left(\frac{\inner{Au,(A-I)u}_2}{\inner{Au,u}_2}\right).
\end{equation}
\end{lemma}

\begin{proof}
Put $\rho=\inner{Au,(A-I)u}_2/\inner{Au,u}_2\ge0$. Cauchy--Schwarz gives
\[
\inner{Au,u}_2^2\le\norm u_2^2\norm{Au}_2^2=(1+\rho)\norm u_2^2\inner{Au,u}_2,
\]
so $\inner{Au,u}_2-\norm u_2^2\le\rho\norm u_2^2$. Combining \eqref{Eqn: Free centered cubic pairing} and \eqref{Ineq: Free cubic bound} therefore yields
\begin{equation}\label{Ineq: Free centered scalar bound}
\frac{\operatorname{Re}\inner{Ax,x^2}_2}{2\inner{Ax,x}_2}=\tau(x)+\frac{\operatorname{Re}\inner{Au,u^2}_2}{2\inner{Au,u}_2}\le\tau(x)+\norm u_2\left(\sqrt{2\rho}+\frac{\rho^{3/2}}2\right).
\end{equation}
Since $\tau(x)^2+\norm u_2^2=\norm x_2^2$, another Cauchy--Schwarz inequality gives
\begin{equation}\label{Ineq: Free scalar polynomial}
\tau(x)+\norm u_2\left(\sqrt{2\rho}+\frac{\rho^{3/2}}2\right) \le\norm x_2\sqrt{1+2\rho+\sqrt2\,\rho^2+\tfrac14\rho^3}.
\end{equation}
For $\rho>0$,
\begin{equation}\label{Ineq: Free strict exponential comparison}
e^{2\rho}\ge1+2\rho+2\rho^2+\tfrac43\rho^3 >1+2\rho+\sqrt2\,\rho^2+\tfrac14\rho^3,
\end{equation}
which proves \eqref{Ineq: Free strict scalar comparison}. If $\rho=0$, then \eqref{Ineq: Free centered scalar bound} bounds the ratio by $\tau(x)<\norm x_2$, since $u\ne0$. This proves the remaining case.
\end{proof}
\begin{theorem}\label{Thm: Free polynomial comparison}
For every nonconstant $x\in\mathbb C[\mathbb F_n]$ with $x\ge\varepsilon\mathbf1>0$,
\begin{equation}\label{Ineq: Free homogeneous comparison}
\mathcal Q_{A;A}(x)= \operatorname{Re}\inner{Ax,Ax-x\log x}_2 +\mathcal E_A(x)\log\norm{x}_2>0.
\end{equation}
For a constant $x=c\mathbf1$ with $c>0$, $\mathcal Q_{A;A}(x)=0$.
\end{theorem}
\begin{proof}
Write $u=x-\tau(x)\mathbf1$, $\rho=\inner{Au,(A-I)u}_2/\inner{Au,u}_2$, and $z=\operatorname{Re}\inner{Ax,x^2}_2/(2\inner{Ax,x}_2)$. Then $Ax=Au$ and $\inner{Ax,x}_2=\inner{Au,u}_2>0$. Combining Lemmas~\ref{Lem: Free endpoint transform} and~\ref{Lem: Free scalar comparison} gives
\[
\mathcal Q_{A;A}(x)\ge\inner{Ax,x}_2\left( \frac{\norm{Ax}_2^2}{\inner{Ax,x}_2}-1-\log z+\log\norm x_2 \right)>\inner{Ax,x}_2\left( \frac{\norm{Ax}_2^2}{\inner{Ax,x}_2}-1-\rho\right)=0.
\]
This proves~\eqref{Ineq: Free homogeneous comparison}.
\end{proof}

\begin{proof}[Proof of Theorem~\ref{Thm: Free group LSI}]
Take $x\in\mathbb C[\mathbb F_n]$ with $x\ge\varepsilon\mathbf1>0$ and set $x_t=e^{-tA}x$. The semigroup preserves the positive lower bound and finite Fourier support, so Proposition~\ref{Prop: General interpolation} applies along this curve of Fourier polynomials. Theorem~\ref{Thm: Free polynomial comparison} gives
\begin{equation}\label{Eqn: Free deficit derivative}
\frac{\dd}{\dd t}\mathcal F_A(x_t) =-4\mathcal Q_{A;A}(x_t) \le0.
\end{equation}
For nonconstant $x$, every $x_t$ is nonconstant, so the inequality is strict. Finite Fourier support gives $x_t\to\tau(x)\mathbf1$ in operator norm and $\mathcal E_A(x_t)\to0$. Thus $\Ent_\tau(x_t^2)\to0$ and $\mathcal F_A(x_t)\to0$. By Proposition~\ref{Prop: General interpolation} we have
\[
\mathcal F_A(x)=4\int_0^\infty\mathcal Q_{A;A}(x_t)\,\mathrm dt\ge0.
\]
For a positive polynomial, apply the result to $x+\varepsilon\mathbf1$ and let $\varepsilon\downarrow0$. Continuity of $s^2\log s^2$ at zero gives~\eqref{Ineq: LSI for free group}.

For optimality, take $h=\lambda_{s_1}+\lambda_{s_1}^*$. Then $h$ is bounded and self-adjoint, with $\tau(h)=0$ and $Ah=h$. For sufficiently small real $s$, $\mathbf1+sh$ is positive and
\[
\Ent_\tau((\mathbf1+sh)^2)=2s^2\norm h_2^2+O(s^3), \qquad \mathcal E_A(\mathbf1+sh)=s^2\norm h_2^2.
\]
The ratio tends to $2$ as $s\to0$, proving sharpness.
\end{proof}

\begin{remark}[Choice of interpolation]\label{Rmk: Choice of interpolation}
The auxiliary semigroup is not unique. We first make explicit the logarithmic comparison underlying the cyclic-group Lagrange-multiplier argument in~\cite{yao2026completesolutionoptimalhypercontractivity}, which uses the generator
\[
L_{\mathrm{nn}}=2I-S_n-S_n^*,\qquad (S_nf)(j)=f(j+1), \qquad A\chi_k=\ell_n(k)\chi_k.
\]
For positive $x$ with $\norm x_2=1$ and $n\ge4$, that argument proves
\[
\inner{L_{\mathrm{nn}}x,x^2-2x\log x}_2\ge0,\qquad 2\inner{L_{\mathrm{nn}}x,Ax}_2-\inner{L_{\mathrm{nn}}x,x^2}_2\ge0.
\]
The first inequality follows from the monotonicity of $s^2-2s\log s$, whose derivative is $2(s-1-\log s)\ge0$. For the second, Yao's functional $K_n$ in~\cite[Eq.~(3.1), Theorem~3.4 and Lemma~3.5]{yao2026completesolutionoptimalhypercontractivity} satisfies, in our normalized inner product,
\[
K_n(x)=2\left(2\inner{L_{\mathrm{nn}}x,Ax}_2-\inner{L_{\mathrm{nn}}x,x^2}_2\right).
\]
The cited estimates give $K_n(x)\ge0$ for $\norm x_2<1$; applying them to $rx$ and letting $r\uparrow1$ gives the boundary case. Adding the two inequalities yields
\[
2\mathcal Q_{A;L_{\mathrm{nn}}}(x)=2\operatorname{Re}\inner{L_{\mathrm{nn}}x,Ax-x\log x}_2\ge0.
\]
For $x_t=e^{-tL_{\mathrm{nn}}}x$, homogeneity and Proposition~\ref{Prop: General interpolation} therefore give $\frac{\dd}{\dd t}\mathcal F_A(x_t)=-4\mathcal Q_{A;L_{\mathrm{nn}}}(x_t)\le0$.

For comparison with scalar third-moment methods, consider the averaging semigroup
\begin{equation}\label{Eqn: Depolarizing interpolation}
L_0=I-E_\tau,\qquad E_\tau x=\tau(x)\mathbf1, \qquad e^{-tL_0}x=E_\tau x+e^{-t}(x-E_\tau x).
\end{equation}
For self-adjoint $u$ with $\tau(u)=0$, this particular generator satisfies
\[
\operatorname{Re}\inner{L_0u,u^2}_2=\tau(u^3).
\]
Thus scalar third-moment bounds can be used for this choice of interpolation. This comparison is separate from the estimates involving the generators proved above and below, and is not used in the proofs of the main theorems. Put $S=\norm u_2$ and $Q=\inner{u,(A-I)u}_2^{1/2}$, where $A$ is the generator of the example, normalized to have spectral gap one. The relevant third-moment estimates are
\begin{align}
|\tau(u^3)|&\le3S^2Q &&\text{on }\mathbb Z_n,\ n\ge4, \label{Ineq: Cubic estimation cyclic group}\\
\tau(u^3)&\le3S^2Q+Q^3 &&\text{in the free-group and related examples}. \label{Ineq: Cubic estimation free group}
\end{align}
For a uniformly positive nonconstant $x$ in the polynomial algebra of the example, with $\norm x_2=1$ and $u=x-\tau(x)\mathbf1$, these moment bounds imply the strict comparison
\begin{equation}\label{Eqn: Averaging normalized comparison}
\mathcal Q_{A;L_0}(x)=\inner{Au,u}_2-\tau(ux\log x)>0.
\end{equation}
By homogeneity and Proposition~\ref{Prop: General interpolation}, the deficit $\mathcal F_A(e^{-tL_0}x)$ therefore decreases strictly away from constants. Besides free groups and finite cyclic groups, the examples covered by these third-moment bounds include free products of copies of $\mathbb Z$ and $\mathbb Z_2$, free Gaussian algebras with the number operator, and free orthogonal quantum groups of Kac type with the normalized heat generator; see \cite{xie2026sharplogsobolevinequalitiesfinite,xiezhang2026hypercontractivityfreegroup,zhang2026sharphypercontractivityfreeorthogonal}.
\end{remark}
\section{Crossed products by integer and finite cyclic groups}\label{Sec: Integer crossed products}

Let $G=\mathbb Z$ or $G=\mathbb Z_n$, with $n=2$ or $n\ge4$, and let $\vartheta$ be a normal trace-preserving automorphism of $(\mathcal N,\tau_{\mathcal N})$, with $\vartheta^n=\id$ in the finite case. On $\mathcal M=\mathcal N\rtimes_\vartheta G$, recall the word-length generator and its semigroup:
\[
\mathsf N(U^ka)=\ell(k)U^ka,\qquad e^{-t\mathsf N}(U^ka)=e^{-t\ell(k)}U^ka,
\]
where $\ell(k)=|k|$ for $G=\mathbb Z$ and $\ell(k)=\min\{k,n-k\}$ for $G=\mathbb Z_n$, $0\le k<n$. We distinguish this generator on $L_2(\mathcal M)$ from the scalar word-length generator on $L_2(\widehat G,m)$:
\[
\mathsf N_G\chi_k=\ell(k)\chi_k,\qquad \Dom(\mathsf N_G)=\left\{f\in L_2(\widehat G,m):\sum_k\ell(k)^2|\widehat f(k)|^2<\infty\right\}.
\]
Here $m$ is normalized Haar measure; $L_2(\widehat G,m)$ is the Fourier realization of the group algebra $L_2(\mathcal L(G))$.

We extend the free-group argument to $\mathcal M$ by adapting the operator-valued transference of Chen, Xu and Yin~\cite{MR3079331} to the measure $\nu$. The resulting vector-valued $L_2$ space carries a Hilbert $\mathcal M$-bimodule structure associated with $\mathcal E_{\mathsf N}$, as in~\cite[Section~4.3]{MR2463708}. On this space, we represent the weighted sum of entropies $\frac12\sum_k\ell(k)\Ent_{\tau_{\mathcal N}}(|a_k|^2)$ and estimate the logarithmic pairing $\operatorname{Re}\inner{\mathsf Nx,x\log x}_2$. A resolvent comparison then reduces the crossed-product estimate to a scalar cubic pairing for $\mathsf N_G$. We conclude with scalar twisted crossed products, quantum tori, and matrix algebras.

\subsection{The sharp crossed-product LSI}\label{Sec: Main proof}
We apply the crossed product version of the interpolation scheme in Section~\ref{Sec: Semigroup interpolation}. For a bounded positive Fourier polynomial $x=\sum_kU^ka_k$, define
\begin{equation}\label{Eqn: Crossed coefficient deficit}
\mathcal F_G(x)=2\mathcal E_{\mathsf N}(x) +\sum_k\Ent_{\tau_{\mathcal N}}(|a_k|^2)-\Ent_\tau(x^2),
\end{equation}
and for $x\ne0$, define
\begin{equation}\label{Eqn: Crossed homogeneous comparison}
\mathcal Q_G(x)=\operatorname{Re}\inner{\mathsf Nx,\mathsf Nx-x\log x}_2+\mathcal E_{\mathsf N}(x)\log\norm x_2+\frac12\sum_k\ell(k)\Ent_{\tau_{\mathcal N}}(|a_k|^2).
\end{equation}

\begin{proposition}\label{Prop: Crossed first derivative}
Let $x=\sum_kU^ka_k\ge\varepsilon\mathbf1>0$ be a Fourier polynomial, and write $x_t=e^{-t\mathsf N}x$. Then
\begin{equation}\label{Eqn: Crossed word deficit derivative}
\frac{\dd}{\dd t}\mathcal F_G(x_t)=-4\mathcal Q_G(x_t).
\end{equation}
Moreover, $\mathcal F_G(x_t)\to\mathcal F_G(P_0x)=0$ as $t\to\infty$, and hence
\begin{equation}\label{Eqn: Crossed integrated deficit}
\mathcal F_G(x)=4\int_0^\infty\mathcal Q_G(x_t)\,\dd t.
\end{equation}
In particular, nonnegativity of $\mathcal Q_G$ along the curve implies $\mathcal F_G(x)\ge0$.
\end{proposition}
\begin{proof}
Apply Proposition~\ref{Prop: General interpolation} with $A=L=\mathsf N$ to differentiate $\mathcal F_{\mathsf N}(x_t)$. The coefficients of $x_t$ are $e^{-t\ell(k)}a_k$, and homogeneity gives
\[
\frac{\dd}{\dd t}\sum_k\Ent_{\tau_{\mathcal N}}(|e^{-t\ell(k)}a_k|^2)=-2\sum_k\ell(k)\Ent_{\tau_{\mathcal N}}(e^{-2t\ell(k)}|a_k|^2).
\]
Adding this derivative proves~\eqref{Eqn: Crossed word deficit derivative}. As $t\to\infty$, finite Fourier support gives $x_t\to a_0=P_0x$ in operator norm and $\mathcal E_{\mathsf N}(x_t)\to0$. The entropy $\Ent_\tau(x_t^2)$ converges to $\Ent_{\tau_{\mathcal N}}(a_0^2)$, while the entropy sum $\sum_k\Ent_{\tau_{\mathcal N}}(e^{-2t\ell(k)}|a_k|^2)$ also converges to $\Ent_{\tau_{\mathcal N}}(a_0^2)$. Thus $\mathcal F_G(x_t)\to0$. Integrating~\eqref{Eqn: Crossed word deficit derivative} on $[0,T]$ and letting $T\to\infty$ proves~\eqref{Eqn: Crossed integrated deficit}.
\end{proof}

Both $\mathcal F_G$ and $\mathcal Q_G$ are homogeneous of degree two. In the operator-valued argument below, fix a bounded Fourier polynomial $x=\sum_kU^ka_k\ge\varepsilon\mathbf1>0$ with $x\ne a_0=P_0x$. We first represent the weighted sum of entropies in $\mathcal Q_G$ on the weighted Hilbert bimodule $L_2(\widehat G,\nu;L_2(\mathcal M))$.

\subsubsection{Representation of the entropy terms}
Let
\[
\mathcal H=L_2(\widehat G,\nu;L_2(\mathcal M)),\qquad \inner{v,w}_{\nu,2}=\int_{\widehat G}\tau\bigl(v(\theta)^*w(\theta)\bigr)\,\dd\nu(\theta),
\]
where $\nu$ is the measure in~\eqref{Eqn: Crossed jump measure}. The closed derivation
\[
\delta:\Dom(\mathsf N^{1/2})\subset L_2(\mathcal M)\longrightarrow\mathcal H,\qquad (\delta y)(\theta)=\delta_\theta y=y-\alpha_\theta(y),
\]
satisfies $\delta^*\delta=\mathsf N$ by~\eqref{Eqn: Crossed jump form}. Write $v_x=\delta x$, so that $\norm{v_x}_{\nu,2}^2=\mathcal E_{\mathsf N}(x)$. Its Fourier components are
\[
P_kv_x(\theta)=(1-\chi_k(\theta))U^ka_k.
\]
We define a bounded operator on $\mathcal H$ by right multiplication on these Fourier subspaces:
\begin{equation}\label{Eqn: Crossed ambient coefficient operator}
(C_xv)(\theta)=v(\theta)+\sum_{\substack{k\ne0\\a_k\ne0}}P_k[v(\theta)]\left(\frac{|a_k|}{\norm{a_k}_2}-\mathbf1\right).
\end{equation}

\begin{lemma}\label{Lem: Crossed coefficient representation}
The operator $C_x$ is bounded, positive and self-adjoint on $\mathcal H$, and $v_x$ lies in its support. With the logarithm taken on the support, the following identity holds in the quadratic-form sense:
\begin{equation}\label{Eqn: compute of P_kdeltax for entropy sum}
\inner{v_x,\log C_x\,v_x}_{\nu,2}=\frac12\sum_k\ell(k)\Ent_{\tau_{\mathcal N}}(|a_k|^2).
\end{equation}
\end{lemma}
\begin{proof}
Only finitely many Fourier blocks differ from the identity. Each block is right multiplication by a bounded positive element; under the identification $U^kL_2(\mathcal N)\cong L_2(\mathcal N)$, its positivity follows from $\inner{c,ch}_2=\norm{ch^{1/2}}_2^2\ge0$ for $h\ge0$.\par 

For $a_k\ne0$, the identity $a_k=a_k\mathbf1_{(0,\infty)}(|a_k|)$ places the corresponding component of $v_x$ in the support of $C_x$. The function $s^2|\log(s/\norm{a_k}_2)|$, extended by zero at zero, is bounded on the spectrum of $|a_k|$. Finite Fourier support and~\eqref{Eqn: Crossed jump identities} therefore imply that the logarithmic quadratic form at $v_x$ is absolutely finite. Finally, Fourier orthogonality and the same jump identity yield
\[
\begin{aligned}
\inner{v_x,\log C_x\,v_x}_{\nu,2}
&=\sum_{\substack{k\ne0\\a_k\ne0}}\int_{\widehat G}|1-\chi_k(\theta)|^2\,\dd\nu(\theta)\,\tau_{\mathcal N}\left(|a_k|^2\log\frac{|a_k|}{\norm{a_k}_2}\right)\\
&=\frac12\sum_k\ell(k)\Ent_{\tau_{\mathcal N}}(|a_k|^2).
\end{aligned}
\]
This proves the lemma.
\end{proof}

\subsubsection{The logarithmic pairing}
To estimate $\operatorname{Re}\inner{\mathsf Nx,x\log x}_2$, equip $\mathcal H$ with the bounded left and right actions
\[
(a\cdot v\cdot b)(\theta)=a\,v(\theta)\alpha_\theta(b),\qquad a,b\in\mathcal M,\quad v\in\mathcal H.
\]
For $a,b\in\mathcal M\cap\Dom(\mathsf N^{1/2})$, the identity $\delta(ab)=a\cdot\delta b+(\delta a)\cdot b$ shows that $ab$ belongs to the same bounded form domain. On each $L_2(\mathcal M)$ fiber, let $L_x$ and $R_{\alpha_\theta(x)}$ denote left and right multiplication. Define their average on $\mathcal H$ by
\begin{equation}\label{Eqn: Crossed midpoint operators}
(T_xv)(\theta)=\frac{L_x+R_{\alpha_\theta(x)}}2v(\theta)=\frac{xv(\theta)+v(\theta)\alpha_\theta(x)}2.
\end{equation}

\begin{lemma}\label{Lem: Crossed spectral endpoint}
The operator $T_x$ is bounded, positive and self-adjoint on $\mathcal H$, with $\varepsilon I\le T_x\le\norm{x}_\infty I$. Moreover,
\begin{equation}\label{Ineq: Crossed word midpoint}
\operatorname{Re}\inner{\mathsf Nx,x\log x}_2\le\mathcal E_{\mathsf N}(x)+\inner{v_x,\log T_x\,v_x}_{\nu,2}.
\end{equation}
\end{lemma}
\begin{proof}
The spectral bounds follow from $\varepsilon\mathbf1\le x\le\norm{x}_\infty\mathbf1$. The divided-difference formula~\eqref{Eqn: Endpoint chain rule}, applied to $s\mapsto s\log s$ and the two module actions, gives
\begin{equation}\label{Eqn: Crossed logarithmic jump}
\delta_\theta(x\log x)=\int_0^1\left[I+\log\bigl((1-s)L_x+sR_{\alpha_\theta(x)}\bigr)\right]\delta_\theta x\,\dd s.
\end{equation}
Concavity of the logarithm in the joint spectral calculus yields
\[
\int_0^1\log\bigl((1-s)L_x+sR_{\alpha_\theta(x)}\bigr)\,\dd s\le\log\left(\frac{L_x+R_{\alpha_\theta(x)}}2\right).
\]
Pairing~\eqref{Eqn: Crossed logarithmic jump} with $v_x$ and using $\delta^*\delta=\mathsf N$, we obtain
\[
\begin{aligned}
\operatorname{Re}\inner{\mathsf Nx,x\log x}_2
&=\operatorname{Re}\inner{v_x,\delta(x\log x)}_{\nu,2}\\
&\le\mathcal E_{\mathsf N}(x)+\int_{\widehat G}\inner{\delta_\theta x,\log\left(\frac{L_x+R_{\alpha_\theta(x)}}2\right)\delta_\theta x}_2\,\dd\nu(\theta)\\
&=\mathcal E_{\mathsf N}(x)+\inner{v_x,\log T_x\,v_x}_{\nu,2}.
\end{aligned}
\]
The spectral bounds on $x$ make all logarithms uniformly bounded. The integrands are consequently dominated by a constant times $\norm{\delta_\theta x}_2^2$, which is $\nu$-integrable.
\end{proof}

Combining Lemmas~\ref{Lem: Crossed coefficient representation} and~\ref{Lem: Crossed spectral endpoint} gives the following bound.
\begin{lemma}
With $v_x$, $C_x$ and $T_x$ as above,
\begin{equation}\label{Ineq: Crossed logarithmic reduction}
\mathcal Q_G(x)\ge\norm{\mathsf Nx}_2^2-\mathcal E_{\mathsf N}(x)+\mathcal E_{\mathsf N}(x)\log\norm x_2+\inner{v_x,(\log C_x-\log T_x)v_x}_{\nu,2}.
\end{equation}
\end{lemma}

\subsubsection{Comparing the two logarithmic forms}
In this subsection we still assume that $x=\sum_kU^ka_k\ge\varepsilon\mathbf1>0$ is a Fourier polynomial. All operators and quadratic forms below act on $\mathcal H$. The coefficient formula gives, for $r>0$,
\[
\norm{(C_x+rI)^{-1}v_x}_{\nu,2}\le\norm{v_x}_{\nu,2}\min\{1,r^{-1}\}.
\]
Indeed, each coefficient satisfies $\norm{a_k(\frac{|a_k|}{\norm{a_k}_2}+r\mathbf1)^{-1}}_2\le\norm{a_k}_2\min\{1,r^{-1}\}$. Together with $T_x\ge\varepsilon I$, this bound justifies the absolutely convergent quadratic-form version of the logarithmic resolvent formula:
\begin{equation}\label{Eqn: Crossed log resolvent representation}
\inner{v_x,(\log C_x-\log T_x)v_x}_{\nu,2} =\int_0^\infty\inner{v_x,\bigl[(T_x+rI)^{-1}-(C_x+rI)^{-1}\bigr]v_x}_{\nu,2}\,\dd r.
\end{equation}
Set $S_r=C_x+rI$ and $R_r=T_x+rI$. For $0<t\le1$, operator convexity of $s\mapsto s^{-1}$ gives
\[
(tR_r+(1-t)S_r)^{-1}\le tR_r^{-1}+(1-t)S_r^{-1}, \qquad \frac{(tR_r+(1-t)S_r)^{-1}-S_r^{-1}}{t}\le R_r^{-1}-S_r^{-1}.
\]
The resolvent identity rewrites the quotient as
\[
\begin{aligned}
\frac{(tR_r+(1-t)S_r)^{-1}-S_r^{-1}}{t}
&=\frac{(tR_r+(1-t)S_r)^{-1}(S_r-(tR_r+(1-t)S_r))S_r^{-1}}{t}\\
&=-(tR_r+(1-t)S_r)^{-1}(R_r-S_r)S_r^{-1}.
\end{aligned}
\]
Letting $t\downarrow0$ in operator norm yields $R_r^{-1}-S_r^{-1}\ge S_r^{-1}(S_r-R_r)S_r^{-1}$. Since $S_r-R_r=C_x-T_x$, integrating the quadratic form at $v_x$ gives
\begin{equation}\label{Ineq: Decomposition of difference of log}
\begin{aligned}
\inner{v_x,(\log C_x-\log T_x)v_x}_{\nu,2} &\ge\int_0^\infty\inner{S_r^{-1}v_x,(C_x-T_x)S_r^{-1}v_x}_{\nu,2}\,\dd r\\
&=\int_0^\infty\inner{S_r^{-1}v_x,C_xS_r^{-1}v_x}_{\nu,2}\,\dd r -\int_0^\infty\inner{S_r^{-1}v_x,T_xS_r^{-1}v_x}_{\nu,2}\,\dd r.
\end{aligned}
\end{equation}
To obtain a bound involving the logarithm of the cubic integral, we also apply this comparison to $T_x/z$, where $z>0$ will be chosen below:
\begin{equation}\label{Ineq: Crossed scaled logarithmic decomposition}
\begin{aligned}
&\inner{v_x,(\log C_x-\log(T_x/z))v_x}_{\nu,2}\\
&\quad\ge\int_0^\infty\inner{S_r^{-1}v_x,C_xS_r^{-1}v_x}_{\nu,2}\,\dd r -\frac1z\int_0^\infty\inner{S_r^{-1}v_x,T_xS_r^{-1}v_x}_{\nu,2}\,\dd r.
\end{aligned}
\end{equation}
We estimate these two integrals separately. For the first integral, $\int_0^\infty c(c+r)^{-2}\,\dd r=1$ for $c>0$ gives
\[
\int_0^\infty C_x(C_x+rI)^{-2}\,\dd r=\mathbf1_{(0,\infty)}(C_x)
\]
in the strong operator topology. This is the support projection of $C_x$, and it fixes $v_x$. Therefore
\begin{equation}\label{Eqn: First part of difference of Log}
\int_0^\infty\inner{S_r^{-1}v_x,C_xS_r^{-1}v_x}_{\nu,2}\,\dd r =\norm{v_x}_{\nu,2}^2=\mathcal E_{\mathsf N}(x).
\end{equation}
For the second integral, we identify the resolvent vectors coefficientwise and use the following weighted pairing estimate.

\begin{lemma}\label{Lem: Crossed resolvent cubic representation}
Retain the fixed $x$, the vector $v_x$ and the operators $C_x,T_x$ defined above. Set
\[
u=\sum_{k\ne0}\norm{a_k}_2\chi_k.
\]
Then $u$ is a real scalar polynomial on $\widehat G$ with $\int_{\widehat G}u\,\dd m=0$. For $r>0$, define
\begin{equation}\label{Eqn: Crossed filtered element}
w_r=\sum_{\substack{k\ne0\\a_k\ne0}}U^ka_k\left(\frac{|a_k|}{\norm{a_k}_2}+r\mathbf1\right)^{-1}.
\end{equation}
Then $w_r=w_r^*$, $P_0w_r=0$ and $w_r\in\mathcal M\cap\Dom(\mathsf N)$. Moreover, with $S_r=C_x+rI$, the following integrals are finite, strictly positive and satisfy
\begin{align}
\int_0^\infty\inner{S_r^{-1}v_x,T_xS_r^{-1}v_x}_{\nu,2}\,\dd r&=\int_0^\infty\int_{\widehat G}\tau\bigl(x(\delta_\theta w_r)^2\bigr)\,\dd\nu(\theta)\,\dd r\label{Eqn: Crossed filtered cubic z}\\
&\le\norm{a_0}_2\mathcal E_{\mathsf N}(x)+\frac12\inner{\mathsf N_Gu,u^2}_2.\label{Ineq: Crossed coefficient cubic bound}
\end{align}
\end{lemma}
\begin{proof}
Since $x=x^*$, one has $P_{-k}x=(P_kx)^*$ and $\norm{a_{-k}}_2=\norm{a_k}_2$. Polar decomposition in~\eqref{Eqn: Crossed filtered element} gives $P_{-k}w_r=(P_kw_r)^*$, so $w_r=w_r^*$. Its Fourier support is finite and excludes zero. The coefficient formulas for $C_x$ and $v_x$ give directly
\[
(S_r^{-1}v_x)(\theta)=\sum_{\substack{k\ne0\\a_k\ne0}}(1-\chi_k(\theta))U^ka_k\left(\frac{|a_k|}{\norm{a_k}_2}+r\mathbf1\right)^{-1}=\delta_\theta w_r.
\]
The resolvent bound above and $\varepsilon I\le T_x\le\norm{x}_\infty I$ show that the first integral is finite and strictly positive. The definition of $T_x$ and traciality yield
\[
\int_0^\infty\inner{S_r^{-1}v_x,T_xS_r^{-1}v_x}_{\nu,2}\,\dd r=\frac12\int_0^\infty\int_{\widehat G}\tau\bigl((x+\alpha_\theta(x))(\delta_\theta w_r)^2\bigr)\,\dd\nu(\theta)\,\dd r.
\]
Trace invariance gives $\tau(\alpha_\theta(x)(\delta_\theta w_r)^2)=\tau(x(\delta_{-\theta}w_r)^2)$. Symmetry of $\nu$ therefore proves~\eqref{Eqn: Crossed filtered cubic z}.

To estimate the right-hand side of \eqref{Eqn: Crossed filtered cubic z}, expand
\[
x(\delta_\theta w_r)^2 =\left(\sum_{k\in G}U^ka_k\right) \left(\sum_{j\ne0}(1-\chi_j(\theta))P_jw_r\right) \left(\sum_{m\ne0}(1-\chi_m(\theta))P_mw_r\right).
\]
The first sum includes $k=0$. A term has zero trace unless $k+j+m=0$. Taking real parts and applying~\eqref{Eqn: Crossed jump identities}, we obtain
\begin{equation}\label{Eqn: cubic with derivations}
    \int_0^\infty\int_{\widehat G}\tau\bigl(x(\delta_\theta w_r)^2\bigr)\,\dd\nu(\theta)\,\dd r=\sum_{\substack{k+j+m=0\\j,m\ne0}} \frac{\ell(j)+\ell(m)-\ell(k)}2\, \operatorname{Re}\int_0^\infty\tau(U^ka_kP_jw_rP_mw_r)\,\dd r.
\end{equation}
All sums are finite, and the indices are taken modulo $n$ when $G=\mathbb Z_n$. Cauchy--Schwarz in $L_2(\mathcal M)$, followed by the row--column form of noncommutative H\"older and~\eqref{Eqn: P_kw_r estimate}, gives
\begin{equation}\label{Ineq: Holder for cubic derivation}
    \begin{split}
        &\left|\int_0^\infty\tau(U^ka_kP_jw_rP_mw_r)\,\dd r\right| \le\norm{a_k}_2\left\|\int_0^\infty P_jw_rP_mw_r\,\dd r\right\|_2\\
&\qquad \le\norm{a_k}_2 \left\|\int_0^\infty P_jw_r(P_jw_r)^*\,\dd r\right\|_2^{1/2} \left\|\int_0^\infty(P_mw_r)^*P_mw_r\,\dd r\right\|_2^{1/2}\\
    \end{split}
\end{equation}
We now evaluate the two integrals on the right. For $k\ne0$ and $a_k\ne0$, the coefficient squares are
\[
\begin{aligned}
(P_kw_r)^*P_kw_r&=|a_k|^2\left(\frac{|a_k|}{\norm{a_k}_2}+r\mathbf1\right)^{-2},\\
P_kw_r(P_kw_r)^* &=(P_{-k}w_r)^*P_{-k}w_r =|a_{-k}|^2\left(\frac{|a_{-k}|}{\norm{a_k}_2}+r\mathbf1\right)^{-2}.
\end{aligned}
\]
The scalar identity $\int_0^\infty s^2(s/c+r)^{-2}\,\dd r=cs$ for $c>0$ and $s\ge0$ therefore gives
\begin{equation}\label{Eqn: P_kw_r estimate}
\int_0^\infty(P_kw_r)^*P_kw_r\,\dd r=\norm{a_k}_2|a_k|, \qquad \int_0^\infty P_kw_r(P_kw_r)^*\,\dd r=\norm{a_k}_2|a_{-k}|.
\end{equation}
Combining \eqref{Eqn: cubic with derivations}, \eqref{Ineq: Holder for cubic derivation}, and \eqref{Eqn: P_kw_r estimate} gives
\[
    \int_0^\infty\int_{\widehat G}\tau\bigl(x(\delta_\theta w_r)^2\bigr)\,\dd\nu(\theta)\,\dd r\leq \sum_{\substack{k+j+m=0\\j,m\ne0}} \frac{\ell(j)+\ell(m)-\ell(k)}{2}\norm{a_k}_2\norm{a_j}_2\norm{a_m}_2.
\]
For $k=0$, one has $m=-j$, so the corresponding part of the upper bound is
\[
\sum_{j\ne0}\ell(j)\norm{a_0}_2\norm{a_j}_2^2 =\norm{a_0}_2\mathcal E_{\mathsf N}(x).
\]
For $k,j,m\ne0$, invariance of the sum under permutations of the three indices gives
\[
\sum_{\substack{k+j+m=0\\k,j,m\ne0}} \frac{\ell(j)+\ell(m)-\ell(k)}2\norm{a_k}_2\norm{a_j}_2\norm{a_m}_2=\frac12\sum_{\substack{k+j+m=0\\k,j,m\ne0}} \ell(k)\norm{a_k}_2\norm{a_j}_2\norm{a_m}_2 =\frac12\inner{\mathsf N_Gu,u^2}_2.
\]
The last equality is Fourier orthogonality on $L_2(\widehat G,m)$; the symmetry $\norm{a_{-k}}_2=\norm{a_k}_2$ makes $u$ real. Combining the two contributions proves~\eqref{Ineq: Crossed coefficient cubic bound}.
\end{proof}

To bound the cubic integral, we use the following scalar Fourier estimate. Its proof in Section~\ref{Sec: Weighted cubic estimates} is a direct calculation that retains the word-length weights.

\begin{lemma}\label{Lem: Cyclic weighted cubic bound}
Let $G=\mathbb Z$ or $G=\mathbb Z_n$, where $n=2$ or $n\ge4$. For every real trigonometric polynomial $u$ on $\widehat G$ satisfying $\int_{\widehat G}u\,dm=0$, with $\mathsf N_G\chi_k=\ell(k)\chi_k$, one has
\begin{equation}\label{Ineq: Cyclic weighted cubic bound}
\bigl|\inner{\mathsf N_Gu,u^2}_2\bigr|^2 \le4\norm u_2^2\left(\norm{\mathsf N_Gu}_2^4 -\inner{u,\mathsf N_Gu}_2^2\right).
\end{equation}
\end{lemma}

We can now prove the homogeneous comparison for crossed products.
\begin{proposition}\label{Thm: Crossed polynomial comparison}
Let $x=\sum_kU^ka_k\ge\varepsilon\mathbf1>0$ be a bounded Fourier polynomial in $\mathcal M$. If $x\ne P_0x$, then
\begin{equation}\label{Ineq: Crossed product homogeneous comparison}
\mathcal Q_G(x)>0.
\end{equation}
If $x=P_0x$, then $\mathcal Q_G(x)=0$.
\end{proposition}
\begin{proof}
Assume $x\ne P_0x$ and put
\[
 \kappa=\frac{\norm{\mathsf Nx}_2^2}{\mathcal E_{\mathsf N}(x)}\ge1.
\]
The last inequality follows from $\ell(k)\ge1$ for $k\ne0$. Fourier orthogonality gives
\[
\norm u_2=\norm{x-a_0}_2,\qquad \inner{\mathsf N_Gu,u}_2=\mathcal E_{\mathsf N}(x),\qquad \norm{\mathsf N_Gu}_2=\norm{\mathsf Nx}_2.
\]
Choose the scaling parameter in~\eqref{Ineq: Crossed scaled logarithmic decomposition} to be
\begin{equation}\label{Eqn: Crossed spectral z}
z=\frac{1}{\mathcal E_{\mathsf N}(x)}\int_0^\infty\int_{\widehat G}\tau\bigl(x(\delta_\theta w_r)^2\bigr)\,\dd\nu(\theta)\,\dd r>0.
\end{equation}
By~\eqref{Eqn: First part of difference of Log} and Lemma~\ref{Lem: Crossed resolvent cubic representation},
\[
\inner{v_x,(\log C_x-\log(T_x/z))v_x}_{\nu,2}\ge \mathcal E_{\mathsf N}(x)-\frac1z\int_0^\infty\inner{S_r^{-1}v_x,T_xS_r^{-1}v_x}_{\nu,2}\,\dd r=0.
\]
Since $\log(T_x/z)=\log T_x-(\log z)I$, this proves
\begin{equation}\label{Ineq: Crossed resolvent logarithmic comparison}
\inner{v_x,(\log T_x-\log C_x)v_x}_{\nu,2}\le \mathcal E_{\mathsf N}(x)\log z.
\end{equation}
Combining the two logarithmic representations gives
\begin{equation}\label{Eqn: Crossed endpoint bound}
\operatorname{Re}\inner{\mathsf Nx,x\log x}_2\le \mathcal E_{\mathsf N}(x)(1+\log z)+\frac12\sum_k\ell(k)\Ent_{\tau_{\mathcal N}}(|a_k|^2).
\end{equation}
The filtered cubic estimate and Lemma~\ref{Lem: Cyclic weighted cubic bound} yield
\begin{equation}\label{Ineq: Crossed cubic energy bound}
z\le\norm{a_0}_2+\frac{\inner{\mathsf N_Gu,u^2}_2}{2\mathcal E_{\mathsf N}(x)}\le\norm{a_0}_2+\norm u_2\sqrt{\kappa^2-1}.
\end{equation}
Using $\norm{a_0}_2^2+\norm u_2^2=\norm x_2^2$, Cauchy--Schwarz gives
\begin{equation}\label{Ineq: Crossed spectral z bound}
z\le\sqrt{\norm{a_0}_2^2+\norm u_2^2}\sqrt{1+\kappa^2-1}=\kappa\norm x_2,
\end{equation}
and hence
\begin{equation}\label{Ineq: Crossed logarithmic energy bound}
\mathcal E_{\mathsf N}(x)\log z\le \mathcal E_{\mathsf N}(x)\log\norm x_2+\mathcal E_{\mathsf N}(x)\log\kappa.
\end{equation}
Substituting in~\eqref{Ineq: Crossed logarithmic reduction}, we obtain
\begin{equation}\label{Ineq: Crossed operator reduction}
\mathcal Q_G(x)\ge \mathcal E_{\mathsf N}(x)\left(\kappa-1-\log\frac z{\norm x_2}\right)\ge \mathcal E_{\mathsf N}(x)(\kappa-1-\log\kappa)\ge0.
\end{equation}
If $\kappa>1$, the last expression is strictly positive. If $\kappa=1$, inequality~\eqref{Ineq: Crossed cubic energy bound} gives $z\le\norm{a_0}_2<\norm x_2$, so the first expression in~\eqref{Ineq: Crossed operator reduction} is strictly positive. For $x=P_0x$, every term in~\eqref{Eqn: Crossed homogeneous comparison} vanishes.
\end{proof}

\begin{proof}[Proof of Theorem~\ref{Thm: Sharp crossed product LSI}]
First let $x\ge\varepsilon\mathbf1>0$ be a bounded Fourier polynomial. Proposition~\ref{Thm: Crossed polynomial comparison} and Proposition~\ref{Prop: Crossed first derivative} give
\[
\frac{\dd}{\dd t}\mathcal F_G(x_t)=-4\mathcal Q_G(x_t)\le0.
\]
Integrating as in~\eqref{Eqn: Crossed integrated deficit} gives $\mathcal F_G(x)\ge0$, or equivalently
\begin{equation}\label{Ineq: Crossed coefficient entropy}
\Ent_\tau(x^2)\le\sum_k\Ent_{\tau_{\mathcal N}}(|a_k|^2)+2\mathcal E_{\mathsf N}(x).
\end{equation}
For a bounded positive Fourier polynomial, apply this inequality to $x+\varepsilon\mathbf1$ and let $\varepsilon\downarrow0$. Bounded functional calculus and continuity of $s^2\log s^2$ at zero give $\Ent_\tau((x+\varepsilon\mathbf1)^2)\to\Ent_\tau(x^2)$ and $\Ent_{\tau_{\mathcal N}}((a_0+\varepsilon\mathbf1)^2)\to\Ent_{\tau_{\mathcal N}}(a_0^2)$. The coefficients with $k\ne0$ and $\mathcal E_{\mathsf N}(x)$ are unchanged. Thus~\eqref{Ineq: Crossed coefficient entropy} holds for all such polynomials.

For such polynomials in $\Dom(A^{1/2})$, each coefficient belongs to $\Dom(B^{1/2})$. The LSI on $\mathcal N$ in~\eqref{Ineq: Base LSI for general element} gives
\[
\Ent_\tau(x^2)\le\sum_k\Ent_{\tau_{\mathcal N}}(|a_k|^2)+2\mathcal E_{\mathsf N}(x)\le C_B\mathcal E_{\widehat B}(x)+2\mathcal E_{\mathsf N}(x).
\]
By density of bounded positive Fourier polynomials in $\Dom(A^{1/2})_+$ for the form norm $(\norm{x}_2^2+\mathcal E_A(x))^{1/2}$, the identity~\eqref{Eqn: Crossed form}, and lower semicontinuity of entropy under $L_1$-convergence of the squares, inequality~\eqref{Ineq: Crossed LSI with cyclic group} extends to every $x\in\Dom(A^{1/2})_+$. This density follows by approximating $x^{1/2}$ in $L_4$ by bounded Fourier polynomials and applying $e^{-tA}$ to their squares $x_k^*x_k$.

Since $\mathcal E_A=\mathcal E_{\widehat B}+\mathcal E_{\mathsf N}$, the preceding inequality gives the LSI with constant $\max\{C_B,2\}$. Restriction to $\mathcal N$ gives the lower bound $C_B$. Take $h=(U+U^*)/\norm{U+U^*}_2$ in the scalar group algebra. Then $Ah=h$ and $\tau h=0$, so
\[
\Ent_\tau((\mathbf1+\varepsilon h)^2)=2\varepsilon^2+O(\varepsilon^3), \qquad \mathcal E_A(\mathbf1+\varepsilon h)=\varepsilon^2.
\]
For small real $\varepsilon$, these positive inputs force the lower bound $2$, proving sharpness for either group.
\end{proof}
\begin{proof}[Proof of Corollary \ref{Thm: Sharp LSI on matrix algebra}]
First apply Theorem~\ref{Thm: Sharp crossed product LSI} with $\mathcal N=\mathbb C$, $B=0$, and $G=\mathbb Z_n$ to obtain the optimal scalar LSI constant $2$ for $B_n$. Applying the theorem again with $\mathcal N=\mathcal D_n$ and $B=B_n$ gives the LSI with constant $2$ for $A_n$. Restriction to $\mathcal D_n$ proves sharpness.
\end{proof}

\subsection{Scalar twisted crossed products}\label{Sec: Scalar twisted products}
We extend the sharp LSI to scalar twisted crossed products by $\mathbb Z^d$ through iteration of the one-dimensional result, and then specialize to quantum tori. For background on twisted crossed products and Fourier analysis, see~\cite{MR241994,MR2511867}.

A normalized $2$-cocycle on $\mathbb Z^d$ is a map $\sigma:\mathbb Z^d\times\mathbb Z^d\to\mathbb T$, where $\mathbb T=\{z\in\mathbb C:|z|=1\}$, satisfying
\begin{equation}\label{Eqn: Scalar cocycle}
\sigma(k,\ell)\sigma(k+\ell,m)=\sigma(\ell,m)\sigma(k,\ell+m),\qquad\sigma(0,k)=\sigma(k,0)=1.
\end{equation}
For a normal trace-preserving action $\vartheta$ of $\mathbb Z^d$ on $\mathcal N$, the twisted crossed product $\mathcal M=\mathcal N\rtimes_{\vartheta,\sigma}\mathbb Z^d$ is the von Neumann algebra generated, in its regular representation, by $\mathcal N$ and unitaries $u_k$ satisfying
\begin{equation}\label{Eqn: Scalar twisted covariance}
u_k a u_k^*=\vartheta_k(a),\qquad u_k u_\ell=\sigma(k,\ell)u_{k+\ell}.
\end{equation}
The canonical faithful normal tracial state is determined by
\[
\tau(u_k a)=\mathbf1_{\{k=0\}}\tau_{\mathcal N}(a).
\]
We give a direct choice of scalar phases; compare the classification in~\cite{MR260890,MR281803}. Put $U_i=u_{e_i}$ and choose a real skew-symmetric matrix $\Theta$ such that $e^{2\pi i\Theta_{ij}}=\sigma(e_i,e_j)/\sigma(e_j,e_i)$. Thus $U_iU_j=e^{2\pi i\Theta_{ij}}U_jU_i$. For $k\in\mathbb Z^d$, set $w_k=U_1^{k_1}\cdots U_d^{k_d}$. The cocycle multiplication rule shows that $w_k$ is a scalar multiple of $u_k$, also for negative powers. Moving each $U_i^{\ell_i}$ past the factors $U_j^{k_j}$ with $j>i$ gives
\[
w_kw_\ell=\exp\left(-2\pi i\sum_{i<j}\Theta_{ij}k_j\ell_i\right)w_{k+\ell}.
\]
Define $v_k=\exp(-\pi i\sum_{i<j}\Theta_{ij}k_ik_j)w_k$. Then
\[
v_kv_\ell=\exp\left(\pi i\sum_{i<j}\Theta_{ij}(k_i\ell_j-k_j\ell_i)\right)v_{k+\ell}=e^{\pi i k^{\mathrm T}\Theta\ell}v_{k+\ell}.
\]
The scalar phases leave $v_kav_k^*=\vartheta_k(a)$ and the trace unchanged and preserve the Fourier indices. Relabeling $v_k$ as $u_k$, we therefore use the real skew-symmetric matrix $\Theta=(\Theta_{ij})$ and write
\begin{equation}\label{Eqn: Skew matrix cocycle}
\sigma(k,\ell)=e^{\pi i k^{\mathrm T}\Theta\ell}=\exp\left(\pi i\sum_{m,j=1}^d k_m \ell_j\Theta_{mj}\right),\qquad k,\ell\in\mathbb Z^d.
\end{equation}
In particular, $\sigma(me_i,ne_j)=e^{\pi i mn\Theta_{ij}}$ for coordinate vectors $e_i,e_j$ and integers $m,n$. We use this realization throughout the following computation. Comparing the two orders of multiplication in \eqref{Eqn: Scalar twisted covariance}, we obtain
\begin{equation}\label{Eqn: Scalar commutation phase}
u_ku_\ell=\frac{\sigma(k,\ell)}{\sigma(\ell,k)}u_\ell u_k=e^{2\pi i k^{\mathrm T}\Theta\ell}u_\ell u_k.
\end{equation}

We recall B\'edos's decomposition theorem for cocycle crossed products~\cite[Theorem~1]{MR1124823}.

\begin{theorem}\label{Thm: Bedos decomposition}
Let $1\to H\to G\to K\to1$ be an exact sequence of discrete groups, and let $(\eta,\omega)$ be a cocycle action of $G$ on a von Neumann algebra $\mathcal N$. Write $(\eta_H,\omega_H)$ for its restriction to $H$. There is a cocycle action $(\beta,\omega_K)$ of $K$ on $\mathcal N\rtimes_{\eta_H,\omega_H}H$ for which
\[
\mathcal N\rtimes_{\eta,\omega}G\cong\bigl(\mathcal N\rtimes_{\eta_H,\omega_H}H\bigr)\rtimes_{\beta,\omega_K}K.
\]
\end{theorem}

In our setting, put $U_j=u_{e_j}$, $\mathcal M_0=\mathcal N$, and $\mathcal M_j=W^*(\mathcal N,U_1,\ldots,U_j)$. Apply Theorem~\ref{Thm: Bedos decomposition} to $\mathbb Z^{j-1}\subset\mathbb Z^j$ with the coordinate section $n\mapsto ne_j$. Since \eqref{Eqn: Skew matrix cocycle} is trivial on $\mathbb Ze_j$, the induced quotient cocycle is trivial. Thus, with $\beta_j=\operatorname{Ad}(U_j)|_{\mathcal M_{j-1}}$, we obtain the normal trace-preserving identification
\begin{equation}\label{Eqn: Scalar twisted iteration}
\mathcal M_j\cong\mathcal M_{j-1}\rtimes_{\beta_j}\mathbb Z.
\end{equation}
This identification preserves the Fourier indices.

For the intermediate direct-sum operators $A_j(u_ka)=u_k(B+|k|_1)a$, $k\in\mathbb Z^j$, with $A_0=B$, the action $\beta_j$ applies $\vartheta_{e_j}$ to each Fourier coefficient and multiplies it by a scalar phase. Hence \eqref{Eqn: Pairwise Covariance} implies that $\beta_j$ preserves $\Dom(A_{j-1})$ and commutes with $A_{j-1}$ there. Since coordinate lengths add and scalar phases preserve the coefficient generator, the one-dimensional construction iterates as follows.

\begin{lemma}\label{Lem: Scalar twisted compatible generator}
Let $B$ be a symmetric completely Markov generator on $L_2(\mathcal N)$, let $\vartheta$ be a normal trace-preserving action of $\mathbb Z^d$ satisfying \eqref{Eqn: Pairwise Covariance}, and let $\sigma$ be a normalized scalar $2$-cocycle. On $\mathcal M=\mathcal N\rtimes_{\vartheta,\sigma}\mathbb Z^d$, the direct-sum operator $A(u_ka)=u_k(B+|k|_1)a$, with its natural direct-sum domain, is a symmetric completely Markov generator. Its semigroup satisfies
\begin{equation}\label{Eqn: Scalar twisted Markov semigroup}
e^{-tA}(u_ka)=e^{-t|k|_1}u_ke^{-tB}(a),\qquad t\geq0.
\end{equation}
\end{lemma}

\begin{proof}
For $k\in\mathbb Z^{j-1}$, the commutation relation~\eqref{Eqn: Scalar commutation phase} gives
\[
\beta_j(u_ka)=e^{2\pi i e_j^{\mathrm T}\Theta k}u_k\vartheta_{e_j}(a).
\]
Thus $\beta_j$ preserves each Fourier component and its coefficient norm. The commutation assumption makes it preserve the direct-sum domain of $A_{j-1}$ and commute with $A_{j-1}$ there. Applying Lemma~\ref{Lem: Compatible generator} at each step of~\eqref{Eqn: Scalar twisted iteration} adds the multiplier $|k_j|$ to the previous generator and proves~\eqref{Eqn: Scalar twisted Markov semigroup} by induction.
\end{proof}

\begin{proof}[Proof of Corollary~\ref{Cor: Zd crossed product LSI}]
Lemma~\ref{Lem: Scalar twisted compatible generator} supplies the Markov generator $A$. We iterate the LSI along \eqref{Eqn: Scalar twisted iteration}, distinguishing two cases.

\noindent\textbf{Case 1: $C_B>0$.} Set
\[
B'=\frac{C_B}2B.
\]
The LSI on $\mathcal N$ has constant $2$ for $B'$, and the same domain invariance and commutation relations hold after this rescaling. Apply Theorem~\ref{Thm: Sharp crossed product LSI} successively to the actions $\beta_j$. Each resulting algebra admits constant $2$, and the final generator is $A'=\widehat{B'}+\mathsf N$. Since $C_B>0$, $A$ and $A'$ have the same form domain. Hence
\[
\Ent_\tau(x^2)\leq2\mathcal E_{A'}(x)=C_B\mathcal E_{\widehat B}(x)+2\mathcal E_{\mathsf N}(x),\qquad x\in\Dom(A^{1/2})_+.
\]

\noindent\textbf{Case 2: $C_B=0$.} Then $\mathcal N=\mathbb C\mathbf1$ and $B=0$: the LSI on $\mathcal N$ gives $\Ent_{\tau_{\mathcal N}}((e^{-tB}a)^2)=0$, so $e^{-tB}a$ is scalar for every $a\in\mathcal N_+$ and $t>0$, and strong continuity at $t=0$ gives the claim. Iterating Theorem~\ref{Thm: Sharp crossed product LSI} from this scalar base yields $\Ent_\tau(x^2)\leq2\mathcal E_{\mathsf N}(x)$, which is \eqref{Ineq: Zd crossed LSI} since $\widehat B=0$.

Since $\mathcal E_A=\mathcal E_{\widehat B}+\mathcal E_{\mathsf N}$, the preceding inequality gives the LSI with constant $\max\{C_B,2\}$. Restriction to $\mathcal N$ gives the lower bound $C_B$. For the lower bound $2$, take $h=(U_1+U_1^*)/\sqrt2$. Then $\tau h=0$, $\norm h_2=1$ and $Ah=h$, so the perturbation calculation for $\mathbf1+\varepsilon h$ in the proof of Theorem~\ref{Thm: Sharp crossed product LSI} applies. Hence the constant is also at least $2$, proving optimality.
\end{proof}

Quantum tori provide the scalar-base case of this construction. We refer to~\cite{MR3079331,MR3778570} for further background on quantum tori. Let $\Theta$ be a real skew-symmetric $d\times d$ matrix. With the normalization \eqref{Eqn: Skew matrix cocycle}, the associated quantum torus $\mathbb T_\Theta^d$ has the canonical tracial realization
\[
\mathbb T_\Theta^d=\mathbb C\rtimes_{\id,\sigma_\Theta}\mathbb Z^d, \qquad\sigma_\Theta(m,\ell)=e^{\pi i m^{\mathrm T}\Theta\ell}.
\]
We denote its canonical Fourier unitaries by $W_m=u_m$, $m\in\mathbb Z^d$, and its canonical trace by $\tau_\Theta$. Thus
\[
W_mW_\ell=e^{\pi i m^{\mathrm T}\Theta\ell}W_{m+\ell},\qquad W_m^*=W_{-m},\qquad \tau_\Theta(W_m)=\mathbf1_{\{m=0\}}.
\]
For coordinate unitaries $U_j=W_{e_j}$, we have $U_iU_j=e^{2\pi i\Theta_{ij}}U_jU_i$.

\begin{proof}[Proof of Corollary~\ref{Thm: Sharp LSI on quantum tori}]
Apply Corollary~\ref{Cor: Zd crossed product LSI} with $\mathcal N=\mathbb C$, $B=0$, the trivial action and cocycle $\sigma_\Theta$. Since entropy on $\mathbb C$ vanishes, $C_B=0$, and the resulting generator $AW_m=|m|_1W_m$ has optimal logarithmic Sobolev constant $\max\{0,2\}=2$.
\end{proof}

\section{Cubic estimates involving the generator}\label{Sec: Cubic pairing estimates}
\subsection{The free-group estimate}\label{sec: weighted cubic estimation for free group}

We now prove the Fourier estimate underlying Lemma~\ref{Lem: Free cubic bound}. For Fourier polynomials $v,w$, the Fourier coefficients of their product satisfy the convolution identity
\begin{equation}\label{eq:convolution}
\widehat{vw}(g)=\sum_{pq=g}\widehat v(p)\widehat w(q).
\end{equation}
The order of the factors in this identity is essential. In the matrix calculations below, $\operatorname{Tr}$ denotes the ordinary, unnormalized matrix trace, distinct from the trace $\tau$ on $\mathcal M$. For a finite, possibly rectangular matrix $B=(B_{rs})$, the Hilbert--Schmidt and operator norms are given by
\[
\norm B_{\mathrm{HS}}^2=\operatorname{Tr}(B^*B)=\sum_{r,s}|B_{rs}|^2,\qquad \norm B_{\mathrm{op}}=\sup_{\norm\xi_{\ell_2}=1}\norm{B\xi}_{\ell_2},
\]
where $\norm\xi_{\ell_2}=(\sum_s|\xi_s|^2)^{1/2}$ is the Euclidean norm.

The following projection estimate at a fixed cancellation depth is recorded in \cite[Lemma~2.1]{xiezhang2026hypercontractivityfreegroup}, where it is traced to the proof of Haagerup's inequality \cite{MR520930}. We include the proof for completeness.
\begin{lemma}\label{lem:projection}
Suppose that the Fourier polynomials $v_i$ and $w_j$ are supported in $S_i$ and $S_j$, respectively, where $i,j\ge0$. Their product has no component of length $k$ unless
\begin{equation}\label{eq:triangle}
|i-j|\le k\le i+j,\qquad i+j+k\text{ is even}.
\end{equation}
For every such $k$,
\begin{equation}\label{eq:projection}
\norm{\Pi_k(v_iw_j)}_2\le\norm{v_i}_2\norm{w_j}_2.
\end{equation}
Moreover, with $u_k$ and $b_k$ as in \eqref{Eqn: Free centering}, we have, for $i,j,k\ge0$,
\begin{equation}\label{eq:triple}
|\tau(u_i u_j u_k)|\begin{cases} \le b_i b_j b_k,&\text{if }|i-j|\le k\le i+j\text{ and }i+j+k\text{ is even},\\
=0,&\text{otherwise}.
\end{cases}
\end{equation}
\end{lemma}

\begin{proof}
Consider reduced words $p\in S_i$ and $q\in S_j$ in the Fourier supports of $v_i$ and $w_j$, respectively. Reducing the product $pq$ removes a maximal terminal block $c$ of $p$ and $c^{-1}$ from the beginning of $q$. If $d=|c|$, then
\[
0\le d\le\min(i,j),\qquad k=|pq|=i+j-2d.
\]
This proves the support restriction \eqref{eq:triangle}.

Fix $k$ satisfying \eqref{eq:triangle}, so that the cancellation depth is $d=(i+j-k)/2$. Choose an ordering of each set $S_r$ and form matrices $U\in M_{|S_{i-d}|\times|S_d|}(\mathbb C)$ and $V\in M_{|S_d|\times|S_{j-d}|}(\mathbb C)$ from the Fourier coefficients of $v_i$ and $w_j$. The rows and columns of $U$ are indexed by $\alpha\in S_{i-d}$ and $c\in S_d$, respectively; those of $V$ are indexed by the same $c\in S_d$ and by $\beta\in S_{j-d}$. Set
\[
U_{\alpha,c}=
\begin{cases}
\widehat v_i(\alpha c),&\alpha c\text{ reduced},\\
0,&\text{otherwise},
\end{cases}
\qquad V_{c,\beta}=
\begin{cases}
\widehat w_j(c^{-1}\beta),&c^{-1}\beta\text{ reduced},\\
0,&\text{otherwise}.
\end{cases}
\]
Each word $p\in S_i$ has a unique representation as a reduced concatenation $\alpha c$ with $|\alpha|=i-d$ and $|c|=d$. Similarly, each $q\in S_j$ has a unique representation as a reduced concatenation $c^{-1}\beta$ with $|c|=d$ and $|\beta|=j-d$. Thus each input word determines exactly one matrix entry. All entries corresponding to nonreduced concatenations are zero, so we may restrict the sums to reduced concatenations. Parseval's identity gives
\[
\begin{aligned}
\norm U_{\mathrm{HS}}^2 &=\sum_{\substack{\alpha\in S_{i-d}\\c\in S_d}}|U_{\alpha,c}|^2=\sum_{\substack{\alpha\in S_{i-d},\,c\in S_d\\\alpha c\text{ reduced}}}\abs{\widehat v_i(\alpha c)}^2=\sum_{p\in S_i}|\widehat v_i(p)|^2=\norm{v_i}_2^2,\\
\norm V_{\mathrm{HS}}^2 &=\sum_{\substack{c\in S_d\\\beta\in S_{j-d}}}|V_{c,\beta}|^2=\sum_{\substack{c\in S_d,\,\beta\in S_{j-d}\\c^{-1}\beta\text{ reduced}}}|\widehat w_j(c^{-1}\beta)|^2=\sum_{q\in S_j}|\widehat w_j(q)|^2=\norm{w_j}_2^2.
\end{aligned}
\]
For $\alpha\in S_{i-d}$ and $\beta\in S_{j-d}$ with $\alpha\beta$ reduced, every factorization $pq=\alpha\beta$ with $p\in S_i$ and $q\in S_j$ has cancellation depth $d$ and takes the form $p=\alpha c$, $q=c^{-1}\beta$. Thus matrix multiplication sums over exactly the possible cancelled blocks, and \eqref{eq:convolution} gives
\[
(UV)_{\alpha,\beta}=\sum_{c\in S_d}U_{\alpha,c}V_{c,\beta}=\sum_{\substack{c\in S_d\\\alpha c\text{ reduced}\\c^{-1}\beta\text{ reduced}}}\widehat v_i(\alpha c)\widehat w_j(c^{-1}\beta) =\widehat{v_iw_j}(\alpha\beta).
\]
Every reduced concatenation $\alpha\beta$ has length $i+j-2d=k$, and every word in $S_k$ has exactly one such split after $i-d$ letters. Consequently,
\[
\norm{\Pi_k(v_iw_j)}_2^2=\sum_{\substack{\alpha\in S_{i-d},\,\beta\in S_{j-d}\\\alpha\beta\text{ reduced}}} \left|\sum_{c\in S_d}U_{\alpha,c}V_{c,\beta}\right|^2.
\]
Applying Cauchy--Schwarz and then dropping the restriction that $\alpha\beta$ is reduced, which only adds nonnegative terms, gives
\[
\norm{\Pi_k(v_iw_j)}_2^2\le\sum_{\substack{\alpha\in S_{i-d}\\\beta\in S_{j-d}}} \left(\sum_{c\in S_d}|U_{\alpha,c}|^2\right) \left(\sum_{c\in S_d}|V_{c,\beta}|^2\right)=\norm U_{\mathrm{HS}}^2\norm V_{\mathrm{HS}}^2=\norm{v_i}_2^2\norm{w_j}_2^2.
\]
This proves \eqref{eq:projection}.

To prove \eqref{eq:triple}, suppose first that $|i-j|\le k\le i+j$ and $i+j+k$ is even. Self-adjointness and orthogonality of the Fourier components of different lengths give
\[
\abs{\tau(u_i u_j u_k)}=\abs{\inner{u_k,u_i u_j}_2}=\abs{\inner{u_k,\Pi_k(u_iu_j)}_2}\leq \norm{u_k}_2\norm{\Pi_k(u_iu_j)}_2\leq b_ib_jb_k.
\]
Otherwise, $\Pi_k(u_iu_j)=0$, so $\tau(u_i u_j u_k)=0$.
\end{proof}

\begin{proof}[Proof of Lemma~\ref{Lem: Free cubic bound}]
Let $u$ be a self-adjoint Fourier polynomial with $\tau(u)=0$, so all Fourier sums are finite. Averaging over the three cyclic permutations gives
\begin{equation}\label{eq:cyclic}
\operatorname{Re}\inner{Au,u^2}_2=\operatorname{Re}\sum_{i,j,k\ge1}\frac{i+j+k}{3}\tau(u_i u_j u_k).
\end{equation}
By \eqref{eq:triple}, only triples satisfying \eqref{eq:triangle} contribute to this sum. Every such triple can be written uniquely as
\[
i=r+s,\qquad j=s+t,\qquad k=t+r,\qquad r,s,t\ge0.
\]
For example, $r=(i+k-j)/2$, with analogous formulas for $s$ and $t$. The triangle and parity conditions ensure that $r,s,t$ are nonnegative integers. The weight in \eqref{eq:cyclic} becomes $2(r+s+t)/3$. Since $u_0=0$, we may include triples with a zero index and sum over all $r,s,t\ge0$:
\[
\operatorname{Re}\inner{Au,u^2}_2=\operatorname{Re}\sum_{r,s,t\ge 0}\frac{2(r+s+t)}{3}\tau(u_{r+s} u_{s+t} u_{t+r}).
\]
We distinguish three cases:
\begin{enumerate}
\item If at least two of $r,s,t$ are zero, then $\tau(u_{r+s}u_{s+t}u_{t+r})=0$ because $u_0=0$.
\item If exactly one of $r,s,t$ is zero, there are three possible positions for the zero.
\item If $r,s,t$ are all positive, the bound $b_{r+s}b_{s+t}b_{t+r}$ is unchanged by permutations of $r,s,t$.
\end{enumerate}
Bounding each real part by the corresponding absolute value, applying \eqref{eq:triple}, and using these symmetries, we obtain
\begin{equation}\label{eq:splitmajorant}
\operatorname{Re}\inner{Au,u^2}_2\le2\sum_{r,s\ge1}(r+s)b_r b_s b_{r+s}+2\sum_{r,s,t\ge1}r\,b_{r+s}b_{s+t}b_{t+r}=4\sum_{r,s\ge1}rb_r b_s b_{r+s}+2\sum_{r,s,t\ge1}r\,b_{r+s}b_{s+t}b_{t+r}.
\end{equation}

For the first term in the final expression, Cauchy--Schwarz yields
\[
4\sum_{r,s\ge1}r b_r b_s b_{r+s}\le 4\left(\sum_{r,s\ge1}r b_r^2b_s^2\right)^{1/2} \left(\sum_{r,s\ge1}r b_{r+s}^2\right)^{1/2}.
\]
The first sum under a square root is $\norm u_2^2\inner{Au,u}_2$, while the second is
\[
\sum_{r,s\ge1}r b_{r+s}^2 =\sum_{k\ge2}\left(\sum_{r=1}^{k-1}r\right)b_k^2 =\frac{1}{2}\sum_{k\ge2}k(k-1)b_k^2=\frac{1}{2}\inner{Au,(A-I)u}_2.
\]
Consequently,
\begin{equation}\label{eq:boundary}
4\sum_{r,s\ge1}r b_r b_s b_{r+s} \le 4(\norm{u}_2^2\inner{Au,u}_2)^{1/2}\left( \frac{\inner{Au,(A-I)u}_2}{2} \right)^{\frac{1}{2}}=2\norm{u}_2\sqrt{2\inner{Au,u}_2\inner{Au,(A-I)u}_2}.
\end{equation}

For the sum with $r,s,t\ge1$, choose $N\ge1$ so that $b_k=0$ for $k>N$, and define the real symmetric matrix $M$ and diagonal weight $D$ by
\[
M=(b_{r+s})_{1\le r,s\le N},\qquad D=\operatorname{diag}(1,2,\ldots,N).
\]
Here $M$ is formed from the norms $b_k$, and $D$ records the weight $r$. Expanding the trace gives
\[
\operatorname{Tr}(DM^3) =\sum_{r,s,t=1}^N r\,M_{rs}M_{st}M_{tr} =\sum_{r,s,t\ge1}r\,b_{r+s}b_{s+t}b_{t+r}.
\]
Thus the second term on the right of \eqref{eq:splitmajorant} is $2\operatorname{Tr}(DM^3)$.

Since $M$ is self-adjoint, spectral calculus gives
\[
\norm M_{\mathrm{op}}\le\norm M_{\mathrm{HS}},\qquad M\le\norm M_{\mathrm{op}}I.
\]
The matrix $\abs M=(M^2)^{1/2}$ commutes with $M$, so multiplication on the left and right by $\abs M$ yields
\[
M^3=\abs{M}M\abs{M}\le\norm M_{\mathrm{op}}\abs{M}^2=\norm M_{\mathrm{op}}M^2.
\]
Multiplying on the left and right by $D^{1/2}$ and taking traces gives
\[
2\operatorname{Tr}(DM^3)=2\operatorname{Tr}(D^{\frac{1}{2}}M^3D^{\frac{1}{2}})\le2\norm M_{\mathrm{op}}\operatorname{Tr}(D^{\frac{1}{2}}M^2D^{\frac{1}{2}})\le2\norm M_{\mathrm{HS}}\operatorname{Tr}(DM^2).
\]
The required matrix counts are
\[
\norm M_{\mathrm{HS}}^2=\sum_{r,s=1}^N b_{r+s}^2 =\sum_{k\ge2}(k-1)b_k^2=\inner{Au,u}_2-\norm{u}_2^2,
\]
and
\[
\operatorname{Tr}(DM^2)=\sum_{r,s=1}^N r\,b_{r+s}^2 =\sum_{k\ge2}\frac{k(k-1)}2b_k^2=\frac{\inner{Au,(A-I)u}_2}{2}.
\]
For $2\le k\le N$, there are $k-1$ positive integer pairs with $r+s=k$, and their weights $r$ sum to $k(k-1)/2$. Terms with $k>N$ vanish because $b_k=0$. Substitution yields
\begin{equation}\label{eq:interior}
2\operatorname{Tr}(DM^3)\leq \inner{Au,(A-I)u}_2\sqrt{\inner{Au,u}_2-\norm{u}_2^2}.
\end{equation}
Combining \eqref{eq:splitmajorant}, \eqref{eq:boundary}, and \eqref{eq:interior} proves \eqref{Ineq: Free cubic bound} for self-adjoint Fourier polynomials with trace zero, as stated in Lemma~\ref{Lem: Free cubic bound}.
\end{proof}
\subsection{The integer and finite-cyclic estimates}\label{Sec: Weighted cubic estimates}
We now prove the Fourier estimate underlying Lemma~\ref{Lem: Cyclic weighted cubic bound}. We adapt the Fourier grouping in~\cite[Lemmas~3.2--3.3]{yao2026completesolutionoptimalhypercontractivity} to the word-length weights. We first treat odd cyclic groups, then account for the self-inverse frequency in the even case. For the integer group, only the triples with $i+j=k$ occur.
\begin{proof}[Proof of Lemma~\ref{Lem: Cyclic weighted cubic bound}]
We first expand the cubic term $\inner{\mathsf N_Gu,u^2}_2$. Since $u=\sum_{k\in G}\widehat u(k)\chi_k$, Fourier orthogonality gives
\[
\widehat{u^2}(k)=\sum_{i+j=k}\widehat u(i)\widehat u(j).
\]
As $u$ is real, $\overline{\widehat u(k)}=\widehat u(-k)$. Parseval's identity and symmetry in the three indices therefore yield
\[
\begin{aligned}
\inner{\mathsf N_Gu,u^2}_2&=\sum_{k\in G}\ell(k)\overline{\widehat u(k)}\sum_{i+j=k}\widehat u(i)\widehat u(j)\\
&=\sum_{i+j+k=0}\ell(k)\widehat u(i)\widehat u(j)\widehat u(k)\\
&=\frac13\sum_{i+j+k=0}\bigl(\ell(i)+\ell(j)+\ell(k)\bigr)\widehat u(i)\widehat u(j)\widehat u(k).
\end{aligned}
\]
In this Fourier expansion, the indices belong to $G$, so the convolution and zero-sum conditions are taken modulo $n$ when $G=\mathbb Z_n$. Since $\widehat u(0)=0$, only triples of nonzero frequencies contribute. The positive-index conditions used below are ordinary integer equalities.

For $G=\mathbb Z_n$, put $m=\lfloor n/2\rfloor$ and $b_j=|\widehat u(j)|$ for $1\le j<n/2$; if $n=2m$, put $b_m=|\widehat u(m)|$.

\noindent\textbf{Case $G=\mathbb Z_{2m+1}$ with $m\ge2$.} Reality of $u$ gives
\[
\norm u_2^2=2\sum_{j=1}^mb_j^2,\qquad \inner{u,\mathsf N_Gu}_2=2\sum_{j=1}^mjb_j^2,\qquad \norm{\mathsf N_Gu}_2^2=2\sum_{j=1}^mj^2b_j^2.
\]
Up to permutations and simultaneous sign changes, the nonzero triples whose sum is zero in $G$ are
\[
(i,j,-k),\quad i+j=k,\quad 1\le i,j,k\le m, \qquad\text{and}\qquad (i,j,k),\quad i+j+k=n,\quad 1\le i,j,k\le m.
\]
Grouping conjugate triples and taking absolute values, we obtain
\[
\begin{aligned}
\bigl|\inner{\mathsf N_Gu,u^2}_2\bigr|
&\le 2\sum_{\substack{1\le i,j,k\le m\\i+j=k}}(i+j+k)b_ib_jb_k
+2\sum_{\substack{1\le i,j,k\le m\\i+j+k=n}}k\,b_ib_jb_k\\
&=2\sum_{k=2}^mb_k\left(2k\sum_{i=1}^{k-1}b_ib_{k-i}\right)
+\frac{2n}{3}\sum_{k=1}^mb_k\sum_{\substack{1\le i,j\le m\\i+j=n-k}}b_ib_j\\
&=2\sum_{k=2}^mb_k\left(
2k\sum_{i=1}^{k-1}b_ib_{k-i}
+\frac n3\sum_{\substack{1\le i,j\le m\\i+j=n-k}}b_ib_j
\right)+\frac{2n}{3}b_1b_m^2\\
&=2\sum_{k=2}^mb_kG_k,
\end{aligned}
\]
where
\[
G_k=
\begin{cases}
2k\displaystyle\sum_{i=1}^{k-1}b_ib_{k-i}
+\dfrac n3\displaystyle\sum_{\substack{1\le i,j\le m\\i+j=n-k}}b_ib_j,
&2\le k<m,\\[6pt]
2m\displaystyle\sum_{i=1}^{m-1}b_ib_{m-i}
+\dfrac n3\displaystyle\sum_{\substack{1\le i,j\le m\\i+j=m+1}}b_ib_j
+\dfrac n3b_1b_m,
&k=m.
\end{cases}
\]
Here the factor $n/3$ follows by symmetry in the three indices of the sum with $i+j+k=n$. The term with $k=1$ in that sum is incorporated into $G_m$.

By the Cauchy--Schwarz inequality,
\[
\bigl|\inner{\mathsf N_Gu,u^2}_2\bigr|^2 \le\left(\sum_{k=2}^mk(k-1)b_k^2\right) \left(4\sum_{k=2}^m\frac{G_k^2}{k(k-1)}\right).
\]
Define the weights
\[
d_{ij}=2\bigl(i(i+1)+j(j+1)\bigr) =(i+j)(i+j+2)+(i-j)^2.
\]
For $2\le k\le m$, set
\[
B_k=\sum_{i=1}^{k-1}d_{i,k-i}b_i^2b_{k-i}^2
+\sum_{\substack{1\le i,j\le m\\i+j=n-k}}d_{ij}b_i^2b_j^2.
\]

If $2\le k<m$, weighted Cauchy--Schwarz gives
\[
G_k^2\le\left(
4k^2\sum_{i=1}^{k-1}\frac1{d_{i,k-i}}
+\frac{n^2}{9}\sum_{\substack{1\le i,j\le m\\i+j=n-k}}\frac1{d_{ij}}
\right)B_k.
\]
There are exactly $k-1$ ordered pairs with $i+j=k$, and $d_{i,k-i}\ge k(k+2)$. Hence
\[
4k^2\sum_{i=1}^{k-1}\frac1{d_{i,k-i}} \le\frac{4k(k-1)}{k+2}.
\]
For the pairs with $i+j=n-k$, the first index ranges from $m+1-k$ to $m$, so there are exactly $k$ such pairs. Moreover,
\[
d_{ij}\ge(n-k)(n-k+2).
\]
Since $k\le m<n/2$, we obtain
\[
\frac{n^2}{9}\sum_{\substack{1\le i,j\le m\\i+j=n-k}}\frac1{d_{ij}}
\le\frac{n^2k}{9(n-k)(n-k+2)}
\le\frac{4k}{9}.
\]
Therefore,
\[
G_k^2\le\left(\frac{4k(k-1)}{k+2}+\frac{4k}{9}\right)B_k \le2k(k-1)B_k,\qquad 2\le k<m,
\]
where the last inequality follows from $4/(k+2)\le1$ and $4/9\le k-1$.

We now consider $k=m$. In this case,
\[
B_m=\sum_{i=1}^{m-1}d_{i,m-i}b_i^2b_{m-i}^2
+\sum_{\substack{1\le i,j\le m\\i+j=m+1}}d_{ij}b_i^2b_j^2.
\]
Distributing the additional term $(n/3)b_1b_m$ equally between the ordered pairs $(1,m)$ and $(m,1)$, their coefficients in $G_m$ become $n/2$. The remaining $m-2$ pairs with $i+j=m+1$ have coefficient $n/3$. For these pairs,
\[
d_{ij}\ge(m+1)(m+3)=(n-m)(n-m+2).
\]
Applying weighted Cauchy--Schwarz, we obtain
\[
\begin{aligned}
G_m^2
&\le\left[
\frac{4m(m-1)}{m+2}
+\frac{n^2(m-2)}{9(m+1)(m+3)}
+\frac{2n^2}{4(m+1)(m+3)}
\right]B_m\\
&=\left[
\frac{4m(m-1)}{m+2}
+\frac{n^2}{9(m+1)(m+3)}
\left(m+\frac52\right)
\right]B_m\\
&\le\left[
\frac{4m(m-1)}{m+2}
+\frac49\left(m+\frac52\right)
\right]B_m\\
&\le2m(m-1)B_m.
\end{aligned}
\]
Consequently,
\[
G_k^2\le2k(k-1)B_k,\qquad 2\le k\le m.
\]

The pair families defining $B_2,\ldots,B_m$ are disjoint and cover $\{1,\ldots,m\}^2$ except for $(m,m)$. Therefore,
\[
4\sum_{k=2}^m\frac{G_k^2}{k(k-1)}\le8\sum_{k=2}^mB_k=8\left(\sum_{i,j=1}^md_{ij}b_i^2b_j^2-d_{mm}b_m^4\right)\le8\sum_{i,j=1}^md_{ij}b_i^2b_j^2.
\]
Since
\[
\sum_{i,j=1}^md_{ij}b_i^2b_j^2 =4\left(\sum_{i=1}^mb_i^2\right) \left(\sum_{j=1}^mj(j+1)b_j^2\right),
\]
the preceding Cauchy--Schwarz estimate gives
\[
\begin{aligned}
\bigl|\inner{\mathsf N_Gu,u^2}_2\bigr|^2
&\le8\left(\sum_{k=2}^mk(k-1)b_k^2\right)
\sum_{i,j=1}^md_{ij}b_i^2b_j^2\\
&=32\left(\sum_{i=1}^mb_i^2\right)
\left(\sum_{j=1}^mj(j-1)b_j^2\right)
\left(\sum_{k=1}^mk(k+1)b_k^2\right)\\
&=4\norm u_2^2
\left(\norm{\mathsf N_Gu}_2^2-\inner{u,\mathsf N_Gu}_2\right)
\left(\norm{\mathsf N_Gu}_2^2+\inner{u,\mathsf N_Gu}_2\right)\\
&=4\norm u_2^2
\left(\norm{\mathsf N_Gu}_2^4-\inner{u,\mathsf N_Gu}_2^2\right).
\end{aligned}
\]
This proves the desired estimate for $G=\mathbb Z_{2m+1}$ with $m\ge2$.

\noindent\textbf{Case $G=\mathbb Z_{2m}$ with $m\ge1$.} The self-inverse frequency $m$ occurs only once, so
\[
\norm u_2^2=2\sum_{j=1}^{m-1}b_j^2+b_m^2,\qquad \inner{u,\mathsf N_Gu}_2=2\sum_{j=1}^{m-1}jb_j^2+mb_m^2,\qquad \norm{\mathsf N_Gu}_2^2=2\sum_{j=1}^{m-1}j^2b_j^2+m^2b_m^2.
\]
If $m=1$, then $G=\mathbb Z_2$. The only nonzero Fourier frequency is $1$, so $\mathsf N_Gu=u$ and $u^2$ is constant. Since $u$ has mean zero, we obtain
\[
\inner{\mathsf N_Gu,u^2}_2=0,\qquad \norm{\mathsf N_Gu}_2^4-\inner{u,\mathsf N_Gu}_2^2=0.
\]
Thus the desired inequality holds with equality.

Assume henceforth that $m\ge2$. Up to permutations and simultaneous sign changes, the nonzero triples whose sum is zero in $G$ are
\[
(i,j,-k),\quad i+j=k,\quad 1\le i,j,k\le m, \qquad\text{and}\qquad (i,j,k),\quad i+j+k=n,\quad 1\le i,j,k<m.
\]
The first family includes the triples containing the self-inverse frequency $m$, while the second family contains only frequencies less than $m$. Grouping conjugate triples and taking absolute values, we obtain
\[
\begin{aligned}
\bigl|\inner{\mathsf N_Gu,u^2}_2\bigr|
&\le2\sum_{\substack{1\le i,j,k\le m\\i+j=k}}
(i+j+k)b_ib_jb_k
+2\sum_{\substack{1\le i,j,k<m\\i+j+k=n}}
k\,b_ib_jb_k\\
&=2\sum_{k=2}^mb_k\left(
2k\sum_{i=1}^{k-1}b_ib_{k-i}
\right)
+\frac{2n}{3}\sum_{k=2}^{m-1}b_k
\sum_{\substack{1\le i,j<m\\i+j=n-k}}b_ib_j\\
&=2\sum_{k=2}^mb_kG_k,
\end{aligned}
\]
where
\[
G_k=
\begin{cases}
2k\displaystyle\sum_{i=1}^{k-1}b_ib_{k-i}
+\dfrac n3\displaystyle\sum_{\substack{1\le i,j<m\\i+j=n-k}}b_ib_j,
&2\le k<m,\\[6pt]
2m\displaystyle\sum_{i=1}^{m-1}b_ib_{m-i},
&k=m.
\end{cases}
\]
There is no contribution with $k=1$ in this sum, since $i,j<m$ implies $i+j\le n-2$. Applying Cauchy--Schwarz with weights $k(k-1)$ for $k<m$ and $m(m-1)/2$ for $k=m$, we obtain
\[
\bigl|\inner{\mathsf N_Gu,u^2}_2\bigr|^2 \le \left( \sum_{k=2}^{m-1}k(k-1)b_k^2 +\frac{m(m-1)}2b_m^2 \right)\times \left( 4\sum_{k=2}^{m-1}\frac{G_k^2}{k(k-1)} +\frac{8G_m^2}{m(m-1)} \right).
\]
Define the weights
\[
d_{ij}=2\bigl(i(i+1)+j(j+1)\bigr) =(i+j)(i+j+2)+(i-j)^2.
\]
For $2\le k\le m$, set
\[
B_k=
\begin{cases}
\displaystyle\sum_{i=1}^{k-1}d_{i,k-i}b_i^2b_{k-i}^2
+\displaystyle\sum_{\substack{1\le i,j<m\\i+j=n-k}}d_{ij}b_i^2b_j^2,
&2\le k<m,\\[6pt]
\displaystyle\sum_{i=1}^{m-1}d_{i,m-i}b_i^2b_{m-i}^2,
&k=m.
\end{cases}
\]
For $2\le k<m$, the same weighted Cauchy--Schwarz estimate as in the odd case gives
\[
G_k^2\le2k(k-1)B_k.
\]
For $k=m$, there is only the sum with $i+j=m$. Since $d_{i,m-i}\ge m(m+2)$, we have the stronger bound
\[
G_m^2\le\frac{4m(m-1)}{m+2}B_m\le m(m-1)B_m.
\]
The pair families defining $B_2,\ldots,B_m$ are disjoint and cover $\{1,\ldots,m-1\}^2$. Indeed, the sums $i+j=k$ with $2\le k<m$ cover the pairs with $2\le i+j<m$, the sum defining $B_m$ covers $i+j=m$, and the sums $i+j=n-k$ with $2\le k<m$ cover $m<i+j\le2m-2$. Consequently,
\[
4\sum_{k=2}^{m-1}\frac{G_k^2}{k(k-1)} +\frac{8G_m^2}{m(m-1)}\le8\sum_{k=2}^mB_k=8\sum_{i,j=1}^{m-1}d_{ij}b_i^2b_j^2=32\left(\sum_{i=1}^{m-1}b_i^2\right) \left(\sum_{j=1}^{m-1}j(j+1)b_j^2\right).
\]
Using the quadratic Fourier identities, we have
\[
\begin{aligned}
\sum_{k=2}^{m-1}k(k-1)b_k^2+\frac{m(m-1)}2b_m^2
&=\frac12\left(
\norm{\mathsf N_Gu}_2^2-\inner{u,\mathsf N_Gu}_2
\right),\\
\sum_{i=1}^{m-1}b_i^2
&\le\frac12\norm u_2^2,\\
\sum_{j=1}^{m-1}j(j+1)b_j^2
&\le\frac12\left(
\norm{\mathsf N_Gu}_2^2+\inner{u,\mathsf N_Gu}_2
\right).
\end{aligned}
\]
Substituting these estimates into the preceding Cauchy--Schwarz bound, we conclude
\[
\begin{aligned}
\bigl|\inner{\mathsf N_Gu,u^2}_2\bigr|^2
&\le32\left(
\sum_{k=2}^{m-1}k(k-1)b_k^2+\frac{m(m-1)}2b_m^2
\right)
\left(\sum_{i=1}^{m-1}b_i^2\right)
\left(\sum_{j=1}^{m-1}j(j+1)b_j^2\right)\\
&\le4\norm u_2^2
\left(
\norm{\mathsf N_Gu}_2^2-\inner{u,\mathsf N_Gu}_2
\right)
\left(
\norm{\mathsf N_Gu}_2^2+\inner{u,\mathsf N_Gu}_2
\right)\\
&=4\norm u_2^2
\left(
\norm{\mathsf N_Gu}_2^4-\inner{u,\mathsf N_Gu}_2^2
\right).
\end{aligned}
\]
This proves the desired estimate for $G=\mathbb Z_{2m}$ with $m\ge1$.

\noindent\textbf{Case $G=\mathbb Z$.} Let $u$ be a real mean-zero trigonometric polynomial, and put $b_j=|\widehat u(j)|$ for $j\ge1$. Then
\[
\norm u_2^2=2\sum_{j\ge1}b_j^2,\qquad \inner{u,\mathsf N_Gu}_2=2\sum_{j\ge1}jb_j^2,\qquad \norm{\mathsf N_Gu}_2^2=2\sum_{j\ge1}j^2b_j^2.
\]
All sums have only finitely many nonzero terms.

Up to permutations and simultaneous sign changes, the nonzero triples whose sum is zero in $\mathbb Z$ are
\[
(i,j,-k),\qquad i+j=k,\qquad i,j,k\ge1.
\]
Consequently,
\[
\abs{\inner{\mathsf N_Gu,u^2}_2}
\le
2\sum_{\substack{i,j,k\ge1\\i+j=k}}
(i+j+k)b_ib_jb_k=
4\sum_{k\ge2}k b_k
\sum_{i=1}^{k-1}b_ib_{k-i}=
2\sum_{k\ge2}b_kG_k,
\]
where
\[
G_k=2k\sum_{i=1}^{k-1}b_ib_{k-i},\qquad k\ge2.
\]
By the Cauchy--Schwarz inequality,
\[
\abs{\inner{\mathsf N_Gu,u^2}_2}^2 \le \left(\sum_{k\ge2}k(k-1)b_k^2\right) \left(4\sum_{k\ge2}\frac{G_k^2}{k(k-1)}\right).
\]
Define
\[
d_{ij}=2\bigl(i(i+1)+j(j+1)\bigr)=(i+j)(i+j+2)+(i-j)^2, \qquad B_k=\sum_{i=1}^{k-1}d_{i,k-i}b_i^2b_{k-i}^2, \qquad k\ge2.
\]
We have $d_{i,k-i}\ge k(k+2)$. Another application of the Cauchy--Schwarz inequality therefore gives
\[
G_k^2 \le 4k^2\left(\sum_{i=1}^{k-1}\frac1{d_{i,k-i}}\right)B_k\le \frac{4k(k-1)}{k+2}B_k\le k(k-1)B_k, \qquad k\ge2,
\]
where the last inequality follows from $k+2\ge4$. The pair families defining $B_k$, $k\ge2$, are disjoint and cover all ordered pairs of positive integers. Hence
\[
4\sum_{k\ge2}\frac{G_k^2}{k(k-1)} \le 4\sum_{k\ge2}B_k= 4\sum_{i,j\ge1}d_{ij}b_i^2b_j^2= 16\left(\sum_{i\ge1}b_i^2\right) \left(\sum_{j\ge1}j(j+1)b_j^2\right).
\]
Combining this with the preceding Cauchy--Schwarz estimate, we obtain
\[
\begin{aligned}
\abs{\inner{\mathsf N_Gu,u^2}_2}^2
&\le
16\left(\sum_{i\ge1}b_i^2\right)
\left(\sum_{j\ge1}j(j-1)b_j^2\right)
\left(\sum_{k\ge1}k(k+1)b_k^2\right)\\
&=
2\norm u_2^2
\left(\norm{\mathsf N_Gu}_2^2-\inner{u,\mathsf N_Gu}_2\right)
\left(\norm{\mathsf N_Gu}_2^2+\inner{u,\mathsf N_Gu}_2\right)\\
&=
2\norm u_2^2
\left(\norm{\mathsf N_Gu}_2^4-\inner{u,\mathsf N_Gu}_2^2\right).
\end{aligned}
\]
Thus the integer case satisfies the stronger estimate with constant $2$. In all three cases, we obtain the claimed bound
\[
\abs{\inner{\mathsf N_Gu,u^2}_2}^2 \le 4\norm u_2^2 \left(\norm{\mathsf N_Gu}_2^4-\inner{u,\mathsf N_Gu}_2^2\right).
\]
This proves the lemma.
\end{proof}

\section*{Acknowledgments}
The author thanks Professor Haonan Zhang for helpful discussions on the monotonicity of the logarithmic Sobolev deficit in the free-group case and Professor Simeng Wang for his useful comments on the manuscript. She is also grateful to Dr. Hao Zhang and Dr. Tiantian Zhao for sharing partial results on related problems. The author acknowledges the use of ChatGPT-6 Astra and takes responsibility for the mathematical content and final text.

\begingroup
\emergencystretch=2em
\printbibliography
\endgroup
\addresseshere

\end{document}